\documentclass{amsart}
\usepackage[a4paper,margin=1in]{geometry}

\usepackage{amsmath,amssymb}
\usepackage{amsthm}
\usepackage[abbrev]{amsrefs}
\usepackage{latexsym}
\usepackage[all,cmtip]{xy}
\usepackage{bm}
\usepackage{url}
\usepackage{tikz}
\usepackage{tikz-cd}

\numberwithin{equation}{section}

\newtheorem{thm}{Theorem}[section]
\newtheorem{prop}[thm]{Proposition}
\newtheorem{lem}[thm]{Lemma}
\newtheorem{cor}[thm]{Corollary}
\newtheorem{dfn}[thm]{Definition}
\theoremstyle{definition}
\newtheorem{ex}[thm]{Example}
\newtheorem{rem}[thm]{Remark}

\DeclareMathOperator{\A}{\mathcal{A}}
\DeclareMathOperator{\B}{\mathcal{B}}
\DeclareMathOperator{\D}{\mathcal{D}}
\DeclareMathOperator{\E}{\mathcal{E}}
\DeclareMathOperator{\M}{\mathcal{M}}
\DeclareMathOperator{\N}{\mathcal{N}}
\DeclareMathOperator{\K}{\mathcal{K}}
\DeclareMathOperator{\X}{\mathcal{X}}

\DeclareMathOperator{\Hom}{\mathcal{H}om}

\DeclareMathOperator{\str}{\mathcal{O}_\mathit{X}}
\DeclareMathOperator{\Sym}{Sym}

\begin{document}

\title{Koszul Duality for Symmetric Algebras and Derived Loop Spaces}
\author{Isshin Hisamichi}
\address{Mathematical Institute, Tohoku University, Sendai, Miyagi, 980-8578, Japan}
\email{hisamichi.isshin.p4@dc.tohoku.ac.jp}
\date{}{
\maketitle
\renewcommand{\thepage}{\arabic{page}}

\begin{abstract}
In this paper, we prove Koszul duality for symmetric algebras of connective complexes consisting of locally free sheaves of finite rank, which is not necessarily bounded below. As an application, we construct an equivalence between derived categories of derived loop spaces and the symmetric algebra of the tangent complex for arbitrary separated schemes of finite type over a field $k$ of characteristic zero. This can be regarded as a generalization of the classical duality between derived loop spaces and cotangent bundles by Kapranov in the smooth cases. We also give a geometrical interpretation of our equivalence in terms of derived formal stacks.
\end{abstract}

\tableofcontents

\section{Introduction}

Classically, Koszul duality is a well-known phenomenon which relates different two algebras. Typical example of this is the duality between the derived categories symmetric algebras and exterior algebras discovered by Bernshtein, Gel'fand and Gel'fand in \cite{BGG}. This kind of duality also appears in the context of a derived loop space, a derived algebraic analogue of free loop spaces. For a derived scheme $X$, the derived loop space $\mathcal{L}X$ of $X$ is defined to be $\operatorname{Map}(S^1,X)$ (in the sense of derived schemes). When $X$ is a smooth scheme of finite type over $k$ with $\operatorname{char}k=0$, $\mathcal{L}X$ is isomorphic to $\operatorname{Spec}_{\str} \Sym (\Omega_{X/k}[1])$ by Hochschild-Kostant-Rosenberg (HKR) isomorphism \cite{HKR}. Here, $\Sym(\Omega_{X/k}[1])$ is the exterior algebra over $\Omega_{X/k}$. From this, one can expect that there is a duality between sheaves on the derived loop space $\mathcal{L}X$, or dg-$\Sym(\Omega_{X/k}[1])$-modules, and sheaves on the cotangent bundle $T^*X$, or dg-$\Sym(\Omega_{X/k}^{\vee})$-modules. This is known to be true, which is essentially Kapranov's result \cite{Kap}.

We can obtain this Koszul duality for derived loop spaces by using Mirkovi\'{c} and Riche's Koszul duality for symmetric algebras of complexes consisting of locally free sheaves on a scheme \cite{MR2}. This result can be formulated as follows. Let $k$  be a field and let $X$ be a Noetherian scheme over $k$. Consider a complex
\begin{align*}
\X := (0 \to \E^{-n} \to \cdots \to \E^{-1} \to \E^0 \to 0)
\end{align*}
consisting of locally free sheaves of finite rank, where $\E_i$ has cohomological degree $i$. Then we take the symmetric algebra $\Sym \X$ of $\X$. This is a sheaf of dg-algebras on $X$. Besides, we define an additional grading on $\B$, other than cohomological one, by considering each $\E^i$ to have degree $1$. We call this internal grading. Then, the Koszul dual $\A$ of $\B$ is given by 
\begin{align*}
\A := \Sym(\X^{\vee}[-1]) = \Sym(0 \to (\E^0)^{\vee} \to (\E^{-1})^{\vee} \to \cdots \to (\E^{-n})^{\vee} \to 0),
\end{align*}
where $^\vee$ means dual as $\str$-modules, and each $(\E^{-i})^{\vee}$ has cohomological degree $i+1$ and internal degree $-1$. Under this setting, the Koszul duality is stated as follows:
\begin{thm} (Mirkovi\'{c}, Riche \cite{MR2}) \label{MRD}

There is an equivalence of triangulated categories
\begin{align*}
D_{gr}(\A \operatorname{-Mod}_-) \cong D_{gr}(\B \operatorname{-Mod}_-),
\end{align*}
where $\A \operatorname{-Mod}_-$ is the category of graded dg-$\A$-modules which is bounded above on its internal degree, and $D_{gr}(\A \operatorname{-Mod}_-)$ is its derived category.
\end{thm}
If we take $\X=\Omega_{X/k}[1]$ in Theorem \ref{MRD}, then we obtain the following Koszul duality equivalence
\begin{align*}
D_{gr}(\Sym(\Omega_{X/k}[1])\operatorname{-Mod}_-) \cong D_{gr}(\Sym (\Omega_{X/k}^\vee[-2])\operatorname{-Mod}_-).
\end{align*}

Let us consider generalizing this equivalence to arbitrary separated schemes of finite type over $k$ with $\operatorname{char}k=0$. First, the HKR-isomorphism is generalized to the case when $X$ is not smooth by Ben-Zvi and Nadler \cite{BZ}. The derived loop space $\mathcal{L}X$ can be described by the cotangent complex $\mathbb{L}_{X/k}$ of $X$, which is a generalization of the K\"{a}hler differentials. More explicitly, we have an isomorphism $\mathcal{L}X \cong \operatorname{Spec} \Sym (\mathbb{L}_{X/k}[1])$ when $X$ is a separated, quasi-compact scheme over $k$. From this, if we define $\mathbb{T}_{X/k}:=\Hom_{\str}(\mathbb{L}_{X/k},\str)$, one can expect that there is duality between the derived category of the loop space $D(\mathcal{L}X) = D(\Sym (\mathbb{L}_{X/k}[1]))$ and the derived category of dg-$\Sym \mathbb{T}_{X/k}$-modules $D(\Sym \mathbb{T}_{X/k})$ (with some boundedness conditions). However, most of the cases of this duality cannot be deduced from Mirkovi\'c and Riche's Koszul duality due to the unboundedness of $\mathbb{L}_{X/k}$. Although $\mathbb{L}_{X/k}$ is always connective (that is, concentrated in non-positive degrees), when $X$ is not smooth, the cotangent complex is not bounded below in general. Indeed, Avramov \cite{Avr} proved that if the cotangent complex is bounded in cohomologies, then $X$ is locally complete intesection, so the boundedness of $\mathbb{L}_{X/k}$ is too restrictive.

To solve this problem, we extend Mirkovi\'{c} and Riche's result to more general situations. More specifically, we aim to give the following two directions of generalizations: 
\begin{itemize}
\item The ungraded context.
\item The case where the complex is not necessarily bounded below.
\end{itemize} 
In this paper, we prove the following generalization of Theorem \ref{MRD}.
\begin{thm} (cf. Theorem \ref{thm:main})

Let $X$ be a Noetherian scheme over $k$, and let $\mathcal{X} = (\cdots \to \mathcal{E}^{-n} \to \cdots \to \mathcal{E}^{-1} \to \mathcal{E}^0 \to 0$) be a complex (of possibly infinite length) of locally free sheaves of finite rank over X, where $\mathcal{E}_i$ has cohomological degree $i$. We define $\B:=\Sym \X$ and $\A :=\Sym (\X^{\vee}[-1])$. Then there is an equivalence of triangulated categories
\begin{align*}
D^{+}_{gr}(\mathcal{A}) \cong D^{+}_{gr}(\mathcal{B}).
\end{align*} 
If moreover $\E_0 = 0$, we also have an equivalence
\begin{align*}
D^{+}(\mathcal{A}) \cong D^{+}(\mathcal{B}).
\end{align*} 
\end{thm}
In the statement above, the category $D^{+}(\mathcal{A})$ (resp. $D^{+}_{gr}(\mathcal{A})$ ) denotes the derived category of bounded-below (graded) dg-$\A$-modules.

The proof is done by modifying Mirokovi\'{c} and Riche's argument. First assume that the length of $\X$ is finite. Then, the most critical part of the proof is the construction of functors. In \cite{MR2}, the functors are given by
\begin{align}
D_{gr}(\A-\operatorname{Mod}_-) \to D_{gr}(\B-\operatorname{Mod}_-) ,\  &\M \mapsto \B \otimes_{\str} \M, \\
D_{gr}(\B-\operatorname{Mod}_-) \to D_{gr}(\A-\operatorname{Mod}_-) ,\  &\N \mapsto \A^{\vee} \otimes_{\str} \N,
\end{align}
where the tensor products involves some differentials other than the usual differential of tensor products, called the Koszul differential. To prove that these functors are well-defined, \cite{MR2} uses the internal grading in a critical way, so the proof in \cite{MR2} does not work in the ungraded context. 

To modify this,  we rewrite the functors in the follwing way:
\begin{align}
D^{+}(\mathcal{A}) \to D^{+}(\mathcal{B}),\  &\M \mapsto \K_{\A} \otimes_{\A} \M, \\
D^{+}(\mathcal{B}) \to D^{+}(\mathcal{A}),\  &\N \mapsto \Hom_{\B}(\K_{\B},\N),
\end{align}
where $\K_{\A}$ and $\K_{\B}$ are $(\A,\B)$-bimodule which give ``good'' resolutions of $\str$. Some calculations show that these functors coincide with the original ones (1.1) and (1.2). Then, to check that the funtors (1.3) and (1.4) are well-defined, it is enough to show that $\K_{\A}$ and $\K_{\B}$ are K-projective dg-modules. To prove the K-projectivity of them, we need to show that $\K_{\A}$ and $\K_{\B}$ have filtrations whose successive quotients are free, which can be constructed explicitly by studying the structures of $\K_{\A}$ and $\K_{\B}$.

On the other hand, when the length of $\X$ is infinite, we cannot even define the functors by the method of Mirkovi\'{c} and Riche \cite{MR2}. Specifically, we cannot define the Koszul differentials involved in (1.1) and (1.2). Howerver, we can naturally extend the definition of the Koszul differentials to the case when the length of $\X$ is infinite, by considering an appropriate truncation of $\X$. Then we can prove the equivalence using functors defined similarly as above (1.3), (1.4).

As a consequence, we will apply Theorem 1.2 to obtain the Koszul duality for derived loop spaces of arbitrary schemes of finite type over $k$. Since Theorem 1.2 does not require $\mathbb{L}_{X/k}$ to be bounded below, we can deduce the following Koszul duality for general schemes of finite type over $k$.
\begin{thm} (cf. Theorem \ref{thm:loop}, Corollary \ref{cor:loop})

Let $k$ be a field of characteristic $0$. Let $X$ be a separated scheme of finite type over $k$. Let $\mathbb{L}_{X/k}$ be the cotangent complex of $X$ over $k$, and we set $\mathbb{T}_{X/k} :=\mathbb{R} \! \Hom_{\str}(\mathbb{L}_{X/k},\str)$. There is an equivalence of triangulated categories
\begin{align*}
D^+(\mathcal{L}X) = D^+(\Sym (\mathbb{L}_{X/k}[1])) \cong D^+(\Sym (\mathbb{T}_{X/k}[-2])).
\end{align*}
Moreover, if we consider $\Sym (\mathbb{L}_{X/k}[1])$ as a sheaf of graded dg-algebras by considering the sections of  $\mathbb{L}_{X/k}[1]$ to have internal degree $1$, then there is an equivalence of triangulated categories
\begin{align*}
D_{gr}^+(\mathcal{L}X) \cong D_{gr}^+(\Sym (\mathbb{T}_{X/k}[-2])) \cong D_{gr}^{\searrow}(\Sym \mathbb{T}_{X/k}).
\end{align*}
\end{thm}
Here, the category $D_{gr}^{\searrow}(\Sym \mathbb{T}_{X/k})$ is the derived category of dg-$\Sym \mathbb{T}_{X/k}$-modules $\M$ satisfying $H^i(\M)_j=0$ for $i-2j \ll 0$, where the lower index denotes the internal degree.

Though we can obtain Koszul duality for derived loop spaces using Theorem 1.2, its geometrical interpretation is not clear. Specifically, the dg-algebra $\Sym (\mathbb{T}_{X/k}[-2])$ is not connective, and thus the derived category $D(\Sym (\mathbb{T}_{X/k}[-2]))$ cannot be obtained from derived schemes. However, we can relate this category to derived formal stacks. To see this, we use Lurie's theory of a formal moduli problem \cite{DAGX}, which is generalized to arbitrary schemes by Hennion \cite{Henn}. They constructed a functor of $\infty$-categories
\begin{align*}
\mathcal{F}_X : \operatorname{dgLie}_X \to \operatorname{dSt}_X^f,
\end{align*}
where $\operatorname{dgLie}_X$ denotes the $\infty$-category of dg-Lie algebras over $X$, and $\operatorname{dSt}_X^f$ denotes the $\infty$-category of derived formal stacks over $X$. Moreover, they established equivalence of categories between the derived category of representations of a dg-Lie algebra $\mathfrak{g}$ and the derived category of a derived formal stack $\mathcal{F}_X(\mathfrak{g})$. Using a modified version of this result, we can relate $\Sym (\mathbb{T}_{X/k}[-2])$ and $\mathcal{F}_X(\mathbb{T}_{X/k}[-2])$, where $\mathbb{T}_{X/k}[-2]$ is regarded as a commutative dg-Lie algebra over $X$. As a consequence, we obtain the following equivalences.
\begin{thm}(cf. Corollary \ref{main3})\\
Let $k$ be a field of characteristic $0$. Let $X$ be a separated scheme of finite type over $k$. Let $Y=\mathcal{F}_X(\mathbb{T}_{X/k}[-2])$ be a derived formal stack correspoinding to the commutative dg-Lie algebra $\mathbb{T}_{X/k}[-2]$ over $X$. Then we have the following equivalences of $\infty$-categories:
\begin{align*}
\operatorname{IndCoh}_X\mathcal{L}X &\cong \operatorname{IndCoh}_XY, \\
\operatorname{IndCoh}_{gr,X}\mathcal{L}X &\cong \operatorname{IndCoh}_{gr,X}Y,\\
\D^+(\mathcal{L}X) &\cong \D^+(Y),\\
\D^+_{gr}(\mathcal{L}X) &\cong \D^+_{gr}(Y).
\end{align*}
Moreover, if we define $Y'=\mathcal{F}_X(\mathbb{T}_{X/k})$ by regarding $\mathbb{T}_{X/k}$ as a commutative dg-Lie algebra over $X$, then we also have the following equivalence
\begin{align*}
\operatorname{IndCoh}_{gr,X}\mathcal{L}X &\cong \operatorname{IndCoh}_{gr,X}Y'.
\end{align*}
\end{thm}

Here, $\operatorname{IndCoh}_X$ denotes the subcategory of $\operatorname{IndCoh}$ which is locally generated by $\str$ (for details, see Definition \ref{def:cohx} and Definition \ref{def:indcohx2}).

Let us explain the structure of this paper. In section 2, we collect basic definitions and related results we need to prove our main theorems. In section 3, we give a proof for Theorem 1.2. \S 3.1 and \S 3.2 are devoted to construct $\K_{\A}$ and $\K_{\B}$ and study the properties of them. More specifically, we prove that $\K_{\A}$ and $\K_{\B}$ are both quasi-isomorphic to $\str$ and they are K-projective at each stalk. Although in \S 3.1 we require the complex $\X$ is bounded, in \S 3.2 we treat the case when $\X$ is unbounded as the modified version of the content of \S 3.1. In \S 3.3, we construct the equivalences in Theorem 1.2. In section 4, we assume that $\operatorname{char}k=0$ and discuss applications to derived loop spaces. In \S 4.1, we apply Theorem 1.2 to derived loop spaces and deduce Koszul duality for derived loop spaces of arbitrary schemes of finite type over $k$. In \S 4.2, after reviewing the theory of derived formal stacks, we prove the relation between derived loop spaces and derived formal stacks.\\
\textbf{Acknowledgments.} We sincerely thank professor Isamu Iwanari for suggesting a lot of useful advice for our research.
\section{Preliminaries}

In this section, we give basic definitions and related results needed to prove our main theorems. Throughout this paper, we fix a field $k$ and a Noetherian scheme $X$ over $k$.

\subsection{Sheaves of dg-algebras}
The purpose of this section is to review the definition of dg-algebras and related concepts. Throughout this section, we fix a commutative ring $R$.
\begin{dfn}
A graded $R$-algebra $A=\bigoplus_{n \in \mathbb{Z}} A^n$ equipped with degree $1$ homomorphism of graded $R$-modules $d_A:A\to A$ satisfying the follwing conditions
\begin{itemize}
\item $d_A^2=0$.
\item $d_A(ab) =d_A(a)b + (-1)^pad_A(b)$, for $a \in A^p$ and $b \in A$.
\end{itemize}
is called a differential graded (abbriviated by dg) algebra (over $R$). We say that a dg-algebra $A$ is commutative if $ab = (-1)^{pq}ba$ for any $a, \in A^p$ and $b \in A^q$.

If furthermore dg-algebra $A$ is bigraded as $A =\bigoplus_{(i,j) \in \mathbb{Z}^2}A^i_j$ and the differential $d_A$ is of bidegree $(1,0)$ (i.e. $d_A(A^i_j)\subset A^{i+1}_j$), then we say that $A$ is a graded dg-algebra. For $A^i_j$, we call $i$ the cohomological grading and $j$ the internal grading of $A$.
\end{dfn}
\begin{dfn}
Let $A$ be a dg-algebra (resp. graded dg-algebra). A dg-$A$-module $M$ is a graded (resp. bigraded) left $A$-module endowed with degree $1$ (resp. (1,0)) homomorphism of graded $R$-modules satisfying $d_{M}^2=0$ and $d_{M}(am) = d_{A}(a)m+(-1)^{p}d_{M}(m)$ for $a \in A^p$ and $m \in M$.  A dg-$A$-module $M$ is said to be acyclic if its cohomology is zero, that is, $H(M):= \bigoplus_{i \in \mathbb{Z}}H^i(M)=0$, and $M$ is said to be connective if $M^i=0$ for $i \geq 0$.

Let $M$ and $N$ be dg-$A$-modules. A homomorphism of dg-$A$-modules $f:M \to N$ is a graded (resp. bigraded) $A$-module homomorphism satisfying $f \circ d_{M} = d_{\N}\circ f$.

For a dg-$A$-module $M$ and a homogeneous element $m \in M^p$, let $|m|$ denote the cohomological degree of $m$, that is, $|m|:=p$.
\end{dfn}
\begin{rem}
If $A$ is a commutative dg-algebra, then left dg-$A$-modules naturally have right action of $A$ by defining $ma:=(-1)^{|a||m|}am$. 
\end{rem}
\begin{rem}
Let $R$ be a commutative ring. We can consider $R$ as a dg-algebra concentrated at degree $0$. Then a dg-$R$-module is nothing but a chain complex of $R$-modules.
\end{rem}

Hereafter, we always assume that dg-algebras are commutative.

\begin{dfn}
A sheaf of dg-algebras on $X$ is a pair $(\A,d_{\A})$, where $\A=\bigoplus_{i \in \mathbb{Z}} \A^i$ is a sheaf of graded algebras on $X$, with each $\A^i$ quasi-coherent, and $d_{\A} :\A \to \A$  is a degree $1$ homomorphism of $\str$-modules satisfying the following conditions:
\begin{itemize}
\item $d_{\A}^2=0$.
\item For any local section $a \in \A^p$, $b \in \A^q$, $d_{\A}(ab)=d_{\A}(a)b+(-1)^pd_{\A}(b)$. 
\end{itemize}
A dg-$\A$-module is a pair $(\M,d_{\M})$, where $\M=\bigoplus_{i \in \mathbb{Z}}\M^i$ is a graded left $\A$-module with each $\M^i$ quasi-coherent, and $d_{\M} :\M \to \M$  is a degree $1$ homomorphism of $\str$-modules satisfying the following conditions:
\begin{itemize}
\item $d_{\M}^2=0$.
\item For any local section $a \in \A^p$, $m \in \M^q$, $d_{\M}(am)=d_{\A}(a)m+(-1)^pd_{\M}(m)$. 
\end{itemize}
If furthermore a sheaf of dg-algebras $\A$ is bigraded as $\A =\bigoplus_{(i,j) \in \mathbb{Z}^2}\A^i_j$ and the differential $d_{\A}$ is of bidegree $(1,0)$ (i.e. $d_{\A}(\A^i_j)\subset \A^{i+1}_j$), then we say that $\A$ is a graded dg-algebra. For $\A^i_j$, we call $i$ the cohomological grading and $j$ the internal grading of $A$.
\end{dfn}

We give some notations to related categories.
\begin{dfn}
\  

\begin{enumerate}
\item Let $C(\mathcal{A})$ denote the category of dg-$\mathcal{A}$-modules, and let $\operatorname{Gr-}\A$ be the category of graded $\A$-modules (namely, forgetting differentials). We define $D(\mathcal{A})$ to be a localization of $C(\mathcal{A})$ with respect to quasi-isomorphisms. 

\item The category $C^+(\A)$ (resp. $C^-(\A)$) is the full subcategory of $C(\A)$ consisting of bounded below (resp. bounded above) complexes (that is, dg-$\A$-modules $\M$ satisfying $\M^i=0$ for $i\ll 0$ (resp. $i\gg 0$)) and let $D^+(\mathcal{A})$ (resp. $D^-(\mathcal{A})$) be the full subcategory of $D(\A)$ consisting of objects which have bounded-below (resp. bounded-above) cohomologies.

\item When $\A$ is a sheaf of graded dg-algebras on $X$, let $C_{gr}(\A)$ and $D_{gr}(\A)$ denote the category of graded dg-$\A$-modules and its derived category, respectively. The categories $C_{gr}^+(\A)$ and $D_{gr}^+(\A)$ are defined in a similar way.
\end{enumerate}
\end{dfn}
We also use some related $\infty$-categories.
\begin{dfn}
\ 

\begin{enumerate}
\item Let $\D(\A)$ denote the $\infty$-category of dg-$\A$-modules, and let $\D^+(\mathcal{A})$ (resp. $\D^-(\mathcal{A})$) be the full subcategory of $\D(\A)$ consisting of objects which have bounded-below (resp. bounded-above) cohomologies.
\item The $\infty$-category $\operatorname{Coh}\A$ is the full subcategory of $\D(\A)$ consisting of objects which have bounded and coherent cohomologies.
\item Assume $\A$ is connective. Let $\operatorname{IndCoh} \A$ denote the ind-completion of $\operatorname{Coh}\A$.
\item When $\A$ is a graded dg-algebra, we define $\infty$-categories $\D_{gr}(\A)$, $\D^{\pm}_{gr}(\A)$, $\operatorname{Coh}_{gr}\A$ and $\operatorname{IndCoh}_{gr} \A$ similarly as above.
\end{enumerate}
\end{dfn}
For usual dg-algebras $A$, we also use same notations as above.
\begin{rem}
Assume that $X=\operatorname{Spec} R$ is an affine scheme and let $\A$ be a sheaf of (graded) dg-algebras. Then there are equivalences of categories
\begin{align*}
&C(\A) \cong C(\Gamma(X,\A)), \\
&C_{gr}(\A) \cong C_{gr}(\Gamma(X,\A)).
\end{align*}
Since sheaves of dg-algebras and dg-modules are required to be quasi-coherent in each term, this immediately follows from the equivalence
\begin{align*}
\Gamma(X,-) : \operatorname{Qcoh}(X) \to R\operatorname{-Mod},
\end{align*}
where $\operatorname{Qcoh}(X)$ is a category of quasi-coherent sheaves on $X$ and $R\operatorname{-Mod}$ is a category of $R$-modules.
\end{rem}

We recall some operations on dg-modules. 
\begin{dfn}
Let $A$ be a dg-algebra, and let $M$ and $N$ be dg-$A$-modules. 
\begin{enumerate}
\item Let $n \in \mathbb{Z}$. The $n$-shift $M[n]$ of $M$ is a dg-$A$-module defined as follows:
\begin{itemize}
\item $M[n]^i :=M^{i+n}$.
\item $d_{(M[1])}^i :=(-1)^nd_M^{i+n}$.
\end{itemize} 
\item The tensor product $M \otimes_A N$ has natural dg-$A$-modules strcuture as follows. For $m \in M^p$ and $n \in N^q$, $|m\otimes n| := p+q$ and its differrential $d_{M \otimes_A N}$ is defined by the Leibnitz rule
\begin{align*}
d_{M \otimes_A N}(m \otimes n) = d_M(m) \otimes n + (-1)^p m \otimes d_N(n).
\end{align*}
\item A dg-$A$-module $\operatorname{Hom}_A^{\bullet}(M,N)$ is defined as follows:
\begin{itemize}
\item$ \operatorname{Hom}_A^n(M,N) :=$ \{degree $n$ homomorphisms of graded $A$-modules $f:M \to N$ \}
\item For $f \in \operatorname{Hom}_A^n(M,N)$,
\begin{align*}
d_{\operatorname{Hom}_A^{\bullet}(M,N)}(f) := d_{N} \circ f -(-1)^{n} f \circ d_{M}.
\end{align*}
\end{itemize}
\end{enumerate}
\end{dfn}
As in the case of usual rings, we have the following adjunctions.
\begin{lem}
\ 

\begin{enumerate}
\item  Let $A$ be a dg-algebra, and $M$ be a dg-$A$-module. The functor $M\otimes_A - : C(A) \to C(A)$ is left adjoint to $\operatorname{Hom}_A^{\bullet}(M,-) : C(A) \to C(A)$. Namely, we have functorial isomorphisms
\begin{align*}
\operatorname{Hom}_{C(A)}(M \otimes_A N ,P) \cong \operatorname{Hom}_{C(A)} (N, \operatorname{Hom}_A^{\bullet}(M,P))
\end{align*}
for any $N$, $P \in C(A)$. 
\item Let $A$, $B$ be dg-algebras, and $f :A\to B$ be a homomorphism of dg-algebras. Then the functor $B \otimes_{A} - : C(A) \to C(B)$ is left adjoint to the restriction functor $C(B) \to C(A)$. Namely, we have functorial isomorphisms
\begin{align*}
\operatorname{Hom}_{C(B)} (B \otimes_A M ,N) \cong \operatorname{Hom}_{C(A)}(M,N)
\end{align*}
for any dg-$A$-module $M$ and dg-$B$-module $N$.
\end{enumerate}
\end{lem}
\begin{rem} \label{gr}
When $A$ is a graded dg-algebra, we can also obtain the same kind of adjunction for $C_{gr}(A)$. For graded dg-$A$-modules $M$ and $N$, if we define a graded Hom complex $\operatorname{Hom}_{A,gr}^{\bullet}(M,N)$ by
\begin{align*}
\operatorname{Hom}_{A,gr}^{n}(M,N)_m := \{&\text{bidegree $(n,m)$ homomorphisms of bigraded $A$-modules} \\
&\text{$f:M \to N$ (i.e. $f$ satisfies $f(M^i_j) \subset N^{i+n}_{j+m}$)} \}
\end{align*}
with the same differential as $\operatorname{Hom}_{A}^{\bullet}(M,N)$, then $M \otimes_A - :C_{gr}(A) \to C_{gr}(A)$ is left adjoint to the functor $\operatorname{Hom}_{A,gr}^{\bullet}(M,-) :C_{gr}(A) \to C_{gr}(A)$.
\end{rem} 
Then we will see that the same is true for sheaves of dg-algebras. Let $\A$ be a sheaf of dg-algebras on $X$, and $\M$, $\N$ be dg-$\A$-modules. Then we can define dg-$\A$-modules $\M[1]$ and $\M \otimes_{\A} \N$ in the same way as in the case of usual dg-algebras.

The internal Hom $\Hom_{\A}(\M ,\N)$ is defined as follows.
\begin{itemize}
\item $\Hom_{\A}^n (\M , \N)$ is a sheaf which assigns
\begin{align*}
\{ \text{degree $n$ homomorphisms $\M|_U \to \N|_U$ of $\A|_U$-modules} \}
\end{align*}
for each open set $U \subset X$. 
\item For a local section $f \in \Hom_{\A}^n (\M , \N)$,
\begin{align*}
d_{\Hom_{\A}(\M ,\N)}(f) =d_{\N} \circ f -(-1)^nf \circ d_{\M}.
\end{align*}
\end{itemize}

As in the case of usual dg-algebras, we also have the following adjunctions for sheaves. (cf. \cite[\S 1.2]{Ri})
\begin{lem}
\ 

\begin{enumerate}
\item Let $\A$ be a sheaf dg-algebras on $X$, and let $\M$ be a dg-$\A$-module. The functor $\M \otimes_{\A} - : C(\A) \to C(\A)$ is left adjoint to $\Hom_{\A}(\M,-) : C(\A) \to C(\A)$. Namely, we have functorial isomorphisms
\begin{align*}
\operatorname{Hom}_{C(\A)}(\M \otimes_{\A} \N ,\mathcal{P}) \cong \operatorname{Hom}_{C(\A)} (\N, \Hom_{\A}(\M,\mathcal{P}))
\end{align*}
for any $\N$, $\mathcal{P} \in C(\A)$. 
\item Let $\A$, $\B$ be dg-algebras, and let $f :\A \to \B$ be a homomorphism of dg-algebras. Then the functor $\B \otimes_{\A} - : C(\A) \to C(\B)$ is left adjoint to the restriction functor $C(\B) \to C(\A)$. Namely, we have functorial isomorphisms
\begin{align*}
\operatorname{Hom}_{C(\B)} (\B \otimes_{\A} \M ,\N) \cong \operatorname{Hom}_{C(\A)}(\M,\N)
\end{align*}
for any dg-$\A$-module $\M$ and dg-$\B$-module $\N$.
\end{enumerate}
\end{lem}
\begin{rem}
When $\A$ is a sheaf of graded dg-algebras, we can obtain the same kind of adjunction for $C_{gr}(\A)$ by replacing $\Hom_{\A}$ by graded Hom complex $\Hom_{\A,gr}$, which is defined in the same way as in Remark \ref{gr}.
\end{rem}

Now, we will give a brief review of the notion of the symmetric algebra of a complex.
\begin{dfn}
Let $\X$ be a complex of locally free $\str$-modules. The tensor algebra $T(\X)$ is defined to be $T(\X) = \bigoplus_{n=0}^\infty \X^{\otimes n}$ equipped with natural multiplication and differentials induced by the original complex $\X$, where the tensor product is taken over $\str$.

The symmetric algebra $\Sym \X$ of $\X$ is the quotient sheaf of $T(\X)$ by graded-commutative relations, namely, a sheaf of $T(\X)$ generated by local sections $x \otimes y -(-1)^{|x||y|}y\otimes x$ over $T(\X)$.
\end{dfn}

\begin{rem}
Let $\X$ and $\mathcal{Y}$ be complexes of locally free sheaves, and assume that there is a quasi-isomorphism $\varphi :\X \to \mathcal{Y}$. Then $\varphi$ induces a quasi-isomorphism betweem tensor algebras $T(\X) \to T (\mathcal{Y})$. If moreover $\operatorname{char}k=0$, $\varphi$ also induces a quasi-isomorphism between symmetric algebras $\Sym \X \to \Sym \mathcal{Y}$, since $\Sym \X$ can be expressed as $\Sym \X = \bigoplus_{n \geq 0} (\X^{\otimes n}/\Sigma_n)$, where $\Sigma_n$ denotes the $n$-th symmetric group.
\end{rem}

\begin{ex}
Consider $\X =( 0 \to \E^{-1} \to \E^0 \to 0$), where $\E^0$ is in cohomological degree $0$ and $\E^{-1}$ is in cohomological degree $-1$. Then sections of $\E^0$ are commutative, and those of $\E^{-1}$ are anti-commutative in $\Sym \X$. Hence $\Sym \X = \Sym \E^0 \otimes \wedge^{\bullet} \E^{-1}$, where $\Sym \E^0$ is concentrated in cohomological degree 0, and sections of $\E^{-1}$ is of cohomological degree $-1$. Thus as a complex of $\str$-modules, $\Sym \X$ is expressed as
\begin{align*}
\Sym \X = (0 &\to \Sym \E^0 \otimes \wedge^n \E^{-1} \to \Sym \E^0 \otimes \wedge^{n-1} \E^{-1} \to \cdots \to \Sym \E^0 \otimes \E^{-1} \to \Sym \E^0 \to 0).
\end{align*}

Its ``Koszul dual'' is given by $\Sym (\X^{\vee}[-1]) = \wedge^\bullet (\E^{0})^{\vee} \otimes \Sym ((\E^{-1})^{\vee})$, where $\E^{0}$ has cohomological degree 1, and $\E^{-1}$ has cohomological degree 2. As the classical Koszul duality, symmetric algebras are switched to exterior algebras and vice versa.
\end{ex}

\subsection{K-projective dg-modules and K-flat dg-modules}

To prove our main theorems, we will construct functors between derived categories. To check that these functors are well-defined, we need the notion of K-projectivity and K-flatness.

\begin{dfn}
Let $A$ be a dg-algebra, and $M$ be a dg-$A$-module. M is called K-projective if the functor $\operatorname{Hom}_A^{\bullet}(M,-)$ preserves acyclic dg-$A$-modules, that is, for any dg-$A$-module $N$ satisfying $H(M)=0$, we have $H(\operatorname{Hom}_A^{\bullet}(M,N))=0$.

We say that $M$ is K-flat if the functor $M\otimes_A -$ preserves acyclic dg-$A$-modules.
\end{dfn}
By the adjointness of tensor products and Hom, we can prove the following fact.
\begin{lem} ( \cite[Proposition 10.3.4]{Yek})\label{lem:pf}
Let $A$ be a dg-algebra and let $M$ be a K-projective dg-$A$-module. Then $M$ is a K-flat dg-$A$-module.
\end{lem}

Later, we need to verify that a given dg-module is K-projective. There is useful criterion to check the K-projectivity. In order to state this, we prepare some definitions.
\begin{dfn}
Let $A$ be a dg-algebra over $R$ and let M be a dg-$A$-module. We say that $M$ has a split filtration if there exists a filtration
\begin{align*}
0=F_{-1}M \subset F_0M \subset F_1M \subset \cdots \subset F_pM \subset \cdots 
\end{align*}
satisfying the following two conditions:
\begin{enumerate}
\item $\bigcup_p F_pM=M$,\\
\item For each $p \geq 0$, the exists a complex $K_p$ consisting of free $R$-modules satisfying $F_p/F_{p-1} \cong A \otimes_R K_p$.
\end{enumerate}
We say that $M$ has a cellularly split filtration if $M$ has a split filtration and complexes $K_p$ in the condition (2) has zero differentials for all $p$.
\end{dfn}

It is known that dg-modules which has this kind of filtrations are K-projective:

\begin{lem} \label{lem.spa} ( \cite[Theorem 9.10]{BMR}) Let $A$ be a dg-algebra over $R$, and let $M$ be a dg-$A$-module. If $M$ has a cellularly split filtration, $M$ is $K$-projective over $A$. Moreover if $M$ is bounded below and has a split filtration, then $M$ is K-projective over $A$.
\end{lem}

\section{Koszul Duality Equivalence}

\subsection{Resolutions of $\mathcal{O}_X$}
Let 
\begin{align*}
\X := ( 0 \to \E^{-n} \to \cdots \to \E^{-1} \to \E^0 \to 0 )
\end{align*}
be a complex consisting of locally free sheaves of finite rank on $X$, where $\E^{-i}$ has cohomological degree $-i$ and internal degree $1$. We consider a sheaf of graded dg-algberas 
\begin{align*}
\B := \Sym \X
\end{align*}
and we define its Koszul dual $\A$ to be
\begin{align*}
\A := \Sym (\X^{\vee}[-1]) = \Sym(0 \to (\E^0)^{\vee} \to (\E^{-1})^{\vee} \to \cdots \to (\E^{-n})^{\vee}\to 0),
\end{align*}
where $(\E^{-i})^{\vee}$ has cohomological degree $i+1$ and internal degree $-1$. In this section, we introduce two dg-modules $\K_{\A}$ and $\K_{\B}$ quasi-isomorphic to $\mathcal{O}_X$ which is K-projective at each stalk. We only treat the case when $\X$ has finite length in this section. The general case is discussed in the next section, since this case is a modifed version of this section.

First, for a (not nescessarily graded) dg-$\A$-module $\M$ and a dg-$\B$-module $\mathcal{N}$, we define a new $(\B, \A)$-bimodule $\M \otimes^K \mathcal{N}$. This is, as graded $\str$-module, defined to be $\M \otimes^{K} \mathcal{N} = \M \otimes_{\str} \mathcal{N} $. Its differential is the sum $d_{\M \otimes_{\str}\mathcal{N}} + d'$, where $d_{\M \otimes_{\str}\mathcal{N}}$ is the usual differential of the tensor product of complexes $\M \otimes_{\str} \mathcal{N}$. The second one is the composition
\begin{align*}
\M \otimes_{\str} \mathcal{N} \to \M \otimes_{\str} \mathcal{N} \to \M \otimes_{\str} \X \otimes_{\str} \X^{\vee} \otimes_{\str} \mathcal{N} \to \M \otimes_{\str} \mathcal{N},
\end{align*}
where the first map is sign adjustment, a section $m \otimes n$ of $\M \otimes \mathcal{N}$ is mapped to $(-1)^{|m|}m \otimes n$. The second one is induced by the natural map $\mathcal{O}_X \to \Hom_{\str}(\X,\X) \cong \mathcal{X} \otimes_{\str} \mathcal{X}^{\vee}$. The last one is multiplication, using the natural action of $\mathcal{X} \subset \Sym \X = \B$ on $\M$, and $\mathcal{X}^{\vee} \subset \Sym (\X^{\vee}[-1]) = \A$ on $\mathcal{N}$. To see $d_{\M \otimes_{\str} \N} + d'$ defines a differential, we describe $d'$ more explicitly at each stalk. Let $p \in X$ be any point and $\{x_{\alpha}\}_{\alpha}$ be a basis of $\X_p$ over $\mathcal{O}_{X,p}$ and $\{x_{\alpha}^*\}$ be its dual. Then, the morphism $\mathcal{O}_{X,p} \to \X_p \otimes \X_p^{\vee}$ is given by $1 \mapsto \sum_{\alpha}x_{\alpha} \otimes x_{\alpha}^*$. Hence for any $m \otimes n \in \M_p \otimes_{\str} \N_p$, the Koszul differential $d'$ on the stalk at $p$ can be described as
\begin{align*}
d'_p(m \otimes n) = (-1)^{|m|}\sum_{\alpha} mx_{\alpha} \otimes x_{\alpha}^* n.
\end{align*}
Using this formula, we can show that $(d_{\M \otimes_{\str} \N}+d')^2=0$ by direct calculation (for detail, see Lemma \ref{lem:diff}). Hence $d_{\M \otimes_{\str} \N}+d'$ defines a differential. We call $d'$ the Koszul differential of $\M \otimes^K \N$.

The $(\A,\B)$-bimodule $\N \otimes^K \M$ is defined similarly, switching the roles of $\M$ and $\N$ in the construction above, with slight adjustment of signs. More specifically, the Koszul differential of $\N \otimes^K \M$ is given by the composition
\begin{align*}
\N \otimes_{\str} \mathcal{M} \to \N \otimes_{\str} \mathcal{M} \to \N \otimes_{\str} \X^{\vee} \otimes_{\str} \X \otimes_{\str} \mathcal{M} \to \N \otimes_{\str} \mathcal{M},
\end{align*}
where the first map is sign adjustment, a section $n \otimes m$ of $\N \otimes \mathcal{M}$ is mapped to $-(-1)^{|n|}n \otimes m$. The second one is induced by the natural map $\mathcal{O}_X \to \Hom_{\str}(\X,\X) \cong \mathcal{X}^{\vee} \otimes_{\str} \mathcal{X}$, and the last one is multiplication.

Now let us define a dg-$\A$-module $\A^{\vee}$ and a dg-$\B$-module $\B^{\vee}$ as graded dual of $\str$-modules, that is,
\begin{align*}
&\A^{\vee} :=\Hom_{\str,gr}(\A,\str), \\
&\B^{\vee} :=\Hom_{\str,gr}(\B,\str).
\end{align*}
Its bigrading is given as follows:
\begin{align*}
&\A^{\vee} = \bigoplus_{i,j\in \mathbb{Z}}(\A^{\vee})^i_{j}, \ (\A^{\vee})^i_j := \Hom_{\str}(\A^{-i}_{-j},\str),\\
&\B^{\vee} = \bigoplus_{i,j \in \mathbb{Z}}(\B^{\vee})^i_{j}, \ \ (\B^{\vee})^i_j := \Hom_{\str}(\B^{-i}_{-j},\str),
\end{align*}
where the upper indices are cohomological grading, and the lower ones are internal gradings.
\begin{rem}
Since each term of $\A$ is locally free of finite rank, $\A^{\vee}$ coincides with the ungraded dual $\Hom_{\str}(\A,\str)$. On the other hand, in general $\B^{\vee}$ is a strict sub dg-$\B$-module of $\Hom_{\str}(\B,\str)$. However, if we assume $\E_0=0$, we have $\B^{\vee} = \Hom_{\str}(\B,\str)$.
\end{rem}

Then we define
\begin{align*}
&\K_{\A} := \A \otimes^K \B^{\vee},\\
&\K_{\B} := \B \otimes^K \A^{\vee}.
\end{align*}
Note that both $\K_{\A}$ and $\K_{\B}$ have an internal grading induced by $\A$ and $\B$. By this grading, we have $\K_{\A,p}=0$ for $p>0$ and $\K_{\B,p}=0$ for $p<0$.

Since $\K_{\A}^0 = (\A \otimes^K \B^{\vee})^0 =\str$, there is a natural homomorphism of dg-$\A$-modules $\K_{\A} \to \str$, which is the identity map at $0$-th degree, and zero map othewize. Similarly, since $\K_{\B}^0 = (\B \otimes^K \A^{\vee})^0 =\str$, there is a natural homomorphism of dg-$\B$-modules $\str \to \K_{\B}$ defined in the same way. These two homomorphisms are known to be indeed quasi-isomorphisms:

\begin{prop} \label{qis1} ( \cite[Proposition 1.3.1]{MR2})
\ 
\begin{enumerate}
\item The homomorphism of dg-$\A$-modules $\K_{\A} \to \str$ is a quasi-isomorphism.
\item The homomorphism of dg-$\B$-modules $\str \to \K_{\B}$ is a quasi-isomorphism.
\end{enumerate}
\end{prop}
\begin{rem}
In \cite{MR2}, this proposition is proved in the graded sense. More precisely, we have
\begin{align*}
H^q(\K_{\A,p}) =
\begin{cases}
\str\ \ \ &(p=q=0)\\
0 &(otherwize)
\end{cases}
\end{align*}
for any $p,q \in \mathbb{Z}$.
\end{rem}
\begin{prop} \label{prop:proj1}
Let $p \in X$ be any point. Then we have
\begin{enumerate}
\item $\K_{\A,p}$ is a K-projective dg-module over $\A_p$.
\item $\K_{\B,p}$ is a K-projective dg-module over $\B_p$.
\end{enumerate}
\end{prop}
\begin{proof}
To simplify the notation, we will omit the symbol ``$p$'' in this proof and treat $\A$ and $\B$ as usual dg-algebras, and we define 
\begin{align*}
\K_{\B}(i_0,\cdots,i_n) :=\B \otimes_{\str} S^{i_0}((\E^{0})^{\vee})^{\vee} \otimes_{\str} \cdots \otimes_{\str} S^{i_n}((\E^{-n})^{\vee})^{\vee},
\end{align*}
where $S^i((\E^j)^{\vee})$ is defined by
\begin{align*}
S^i((\E^j)^{\vee}) :=
\begin{cases}
\Sym^i ((\E^j)^{\vee})\ &\text{if $j$ is odd,}\\
\wedge^i (\E^j)^{\vee}\ &\text{if $j$ is even.}
\end{cases}
\end{align*}
for $j=0,-1,\cdots -n$. Then, $\K_{B}$ is decomposed to a direct sum
\begin{align*}
\K_{\B} = \bigoplus_{i_0,\cdots,i_n} \K_{\B}(i_0,\cdots,i_n)
\end{align*}
as an $\str$-module.

In order to prove the claim, by Lemma \ref{lem.spa}, it is enough to find a filtration of dg-submodules of $\K_{\B}$
\begin{align*}
0=F_{-1} \subset F_0 \subset F_1 \subset \cdots \subset F_m \subset \cdots \subset \K_{\B}
\end{align*}
which satisfies that $\bigcup_m F_m = \K_{\B}$ and each subquotient $F_{m+1}/F_m$ is free as $\A$-dg modules.

In order to show how to construct such a filtration, first we demonstrate how the proof goes when $n=1$. In the case of $n=1$, $\K_{\B}$ is the direct sum of the form
\begin{align*}
\K_{\B} = \B \otimes_{\str} \left( \bigoplus_{i,j} \left(\wedge^i(\E^0)^\vee \right)^\vee \otimes_{\str} (\Sym^j(\E^{-1})^\vee)^\vee \right) = \bigoplus_{i,j}\K_{\B}(i,j)
\end{align*}
as a complex of $\mathcal{O}_X$-modules. The differential at $\K_{\B}(i,j)$ is the sum of the following two maps:\\
(1) Differtentials of the usual tensor product $\B \otimes_{\str} \A^{\vee}.$

Due to the Leibnitz rule, $\K_{\B}(i,j)$ is mapped to $\K_{\B}(i+1,j-1)$ by this type of differentials.\\
(2) Koszul differtentials.

By definition, Koszul differential is given by multiplication of certain sections of $\E^0$ and $\E^{-1}$. Hence $\K_{\B}(i,j)$ is mapped to $\K_{\B} (i-1,j) \oplus \K_{\B} (i,j-1).$

In total, differentials of $\K_{\B}$ can be depicted as the following diagram:
\begin{align*}
\xymatrix{
	\K_{\B}(0,0) \\%\ar[r] \ar@{^{(}->}[d] &\operatorname{Hom}_\bullet(X,Y)\\
	\K_{\B}(1,0) \ar@{..>}[u] & \K_{\B}(0,1) \ar[l] \ar@{..>}[lu]\\
	\K_{\B}(2,0) \ar@{..>}[u] & \K_{\B}(1,1) \ar@{..>}[u]\ar@{..>}[lu]\ar[l] & \K_{\B}(0,2) \ar@{..>}[lu]\ar[l]\\
	\ \ar@{..>}[u] & \ \ar@{..>}[u]\ar@{..>}[lu] & \ \ar@{..>}[u]\ar@{..>}[lu] & \ \ar@{..>}[lu]\\
	\  & \vdots & \  & \ 
}
\end{align*}
where the dotted arrows are Koszul differentials, and the solid arrows are the differentials induced by the tensor product. Then, we pick a filtration $\{F_m\}$ so that each $F_m$ are dg-submodules of $\K_{\A}$. Since $\A$-actions are closed at each term, it is enough to define $F_m$ so that there are no arrows going outside of $F_m$. More explicitly, in this case, if we define
\begin{align*}
&F_{-1} = 0, F_0=\K_{\B}(0,0), \\
&F_1=F_0 \oplus \K_{\B}(1,0), F_2=F_1\oplus \K_{\B}(0,1) \\
&F_3=F_2 \oplus  \K_{\B}(2,0), F_4=F_3 \oplus  \K_{\B}(1,1), F_5=F_4 \oplus  \K_{\B}(0,2), \cdots
\end{align*}
then clearly $F_m/F_{m-1}$ is a form of $\K_{\B}(i,j)$, which is a tensor product of $\B$ and some shifts of free $\str$-modules. Hence $\{F_m\}$ gives a desired filtration.

The general cases are entirely similar but rather complicated. As in the construction above, we will look at the differentials at direct summands $\K_{\B}(i_0,\cdots,i_n)$. By the same arguments of $n=1$, outgoing differential at $\K_{\B}(i_0,\cdots,i_n)$ is
\begin{align*}
\K_{\B}(i_0,\cdots,i_n) \to &\K_{\B}(i_0+1,i_1-1,\cdots,i_n) \oplus \K_{\B}(i_0,i_1+1,i_2-1,\cdots,i_n) \\
&\oplus \cdots \oplus \K_{\B}(i_0,\cdots,i_{n-1}+1,i_n-1) \\
&\oplus \K_{\B}(i_0-1,\cdots,i_n) \oplus \K_{\B}(i_0,i_1-1,\cdots,i_n) \\
&\oplus \cdots \oplus \K_{\B}(i_0,\cdots,i_n-1),
\end{align*}
and incoming differential at $\K_{\B}(i_0,\cdots,i_n)$ is
\begin{align*}
 &\K_{\B}(i_0-1,i_1+1,\cdots,i_n) \oplus \K_{\B}(i_0,i_1-1,i_2+1,\cdots,i_n) \\
&\oplus \cdots \oplus \K_{\B}(i_0,\cdots,i_{n-1}+1,i_n-1) \\
&\oplus \K_{\B}(i_0+1,\cdots,i_n) \oplus \K_{\B}(i_0,i_1+1,\cdots,i_n) \\
&\oplus \cdots \oplus \K_{\B}(i_0,\cdots,i_n+1)\\
&\to \K_{\B}(i_0,\cdots,i_n).
\end{align*}
Summarizing this, for two indices $(i_0,\cdots,i_n)$ and $(j_0,\cdots,j_n)$, there is no differential $\K_{\B}(i_0,\cdots,i_n) \to \K_{\B}(j_0,\cdots,j_n)$ if $\sum_{s} i_s = \sum_{s}j_s$ and $(i_0,\cdots,i_n) > (j_0,\cdots,j_n)$ in the lexicographic order, or $\sum_{s} i_s < \sum_{s}j_s$. 

Now we construct the filtration $\{F_m\}$. First, let $F_{-1}=0$ and $F_0 = \K_{\B}(0,\cdots,0)$. Consider the set of indices
\begin{align*}
I_{\ell}:=\{ (i_0,\cdots,i_n) | i_1+\cdots+i_n=\ell , \K_{\B}(i_0,\cdots,i_n) \neq 0\},
\end{align*}
and we define a total order on $\bigcup_{\ell=0}^\infty I_{\ell}$ as follows. For $\bm{i} \in I_{s}$ and $\bm{j} \in I_{t}$,
\begin{align*}
\bm{i} \geq \bm{j} \overset{\text{def.}}{\Longleftrightarrow}
\begin{cases}
s>t,\ \text{or} \\ 
s=t \text{ and } \bm{i} \leq \bm{j} \ \text{in the lexicographic order on}\  I_{s}.
\end{cases}
\end{align*}
Then, we take an order-preserving bijection $\mathbb{Z}_{\geq 0} \ni m \mapsto \bm{i}_m \in \bigcup_{\ell=0}^\infty I_{\ell}$ (such a map exists since each $I_{\ell}$ is finite), and define $F_m:=\bigoplus_{j=0}^m\K_{\B}(\bm{i}_j)$. Then, by the description of differentials, each $F_m$ forms a dg-$\B$-submodule of $\K_{\B}$ and clearly $\{F_m\}$ gives a desired filtration on $\K_{\B}$ since each $\K_{\B}(i_0,\cdots,i_n)$ is a tensor product of $\B$ and a locally free sheaf.

The proof for $\K_{\A}$ is much simpler, since $\K_{\A}$ is bounded below. By Lemma \ref{lem.spa}, in this case, it is enough to construct a filtration $\{F_m\}$ of $\K_{\A}$ such that $F_{m+1}/F_m \cong \A \otimes K_m$, where $K_m$ is a complex consisting of free $\mathcal{O}_X$-modules, which necesarily does not have zero differentials. Now, $\K_{\A}$ can be written as
\begin{align*}
\A \otimes \B_{0}^{\vee} \leftarrow \A \otimes \B_{-1}^{\vee} \leftarrow \A \otimes \B_{-2}^{\vee} \leftarrow \cdots
\end{align*}
where the horizontal arrows stand for Koszul differentials. If we take $F_m=\bigoplus_{i=0}^m \A \otimes \B_{-i}^{\vee}$, then $F_{m+1}/F_m\cong \A \otimes \B_{-(m+1)}^{\vee}$ and $\B_{-(m+1)}^{\vee}$ is a complex consisting of free $\mathcal{O}_X$-modules. Hence $\{F_m\}$ gives a desired filtration. Thus $\K_{\A}$ is K-projective over $\A$.
\end{proof}

\subsection{Complex of Infinite Length}

In this section, we treat the case when the length of $\X$ is infinite. In this case, we need a modification to the discussions in the last section. For example, if $\X$ has infinite length, we cannot obtain the isomorphism $\Hom(\X,\X) \cong \X \otimes_{\str} \X^{\vee}$ due to the unboundedness of $\X$, so the construction in the previous section does not work.

Let 
\begin{align*}
\X:=(\cdots \to \E^{-n} \to\cdots \to \E^{-1} \to \E^0 \to 0)
\end{align*}
be a complex of possibly infinite length consisting of locally free sheaves of finite rank on $X$, where $\E^i$ has cohomological degree $i$ and internal degree $1$. We define a sheaf of graded dg algebras $\B = \Sym \X$ and $\A=\Sym (\X^{\vee}[-1])$. We set $\X^{(n)}$ as the truncation of $\X$:
\begin{align*}
\X^{(n)}:= (0 \to \E^{-n} \to\cdots \to \E^{-1} \to \E^0 \to 0).
\end{align*}
As in the previous section, we will define a $(\B,\A)$-bimodule $\M \otimes^K \mathcal{\N}$ for a dg-$\B$-module $\M$ and a dg-$\A$-module $\N$, but we require that $\M \in C^+(\B)$ or $\N \in C^-(\A)$. 

First consider the case when $\M \in C^+(\B)$, and assume that $\M^i = 0$ for $i < N$. As a graded $\str$-module, we define $\M \otimes^K \mathcal{N}= \M \otimes_{\str} \mathcal{N}$. Its differential is the sum $d_{\M \otimes_{\str} \mathcal{N}} +d'$, but the definition of $d'$ is modified as follows. For a section $m\otimes n$ of $\M \otimes_{\str} \mathcal{N}$, $d'(m \otimes n)$ is defined to be the following compositions
\begin{align*}
\M \otimes_{\str} \mathcal{N} &\to \M \otimes_{\str} \mathcal{N} \to \M \otimes_{\str} \X^{(N-|m|)} \otimes_{\str} (\X^{(N-|m|)})^{\vee} \otimes_{\str} \mathcal{N} \to \M \otimes_{\str} \mathcal{N},
\end{align*}
where the first map is sign adjustment, a section $m \otimes n$ is mapped to $(-1)^{|m|}m \otimes n$. The second one is induced by the natural map 
\begin{align*}
\mathcal{O}_X \to \Hom_{\str}(\X^{(N-|m|)},\X^{(N-|m|)}) \cong \mathcal{X}^{(N-|m|)} \otimes_{\str} (\mathcal{X}^{(N-|m|)})^{\vee}.
\end{align*}
Here, note that $\X^{(N-|m|)}$ is bounded. The last one is multiplication, using the natural action of $\mathcal{X} \subset \Sym \X = \B$ on $\mathcal{M}$, and $\mathcal{X}^{\vee} \subset \Sym (\X^{\vee}[-1]) = \A$ on $\N$.
In fact, this defines a differential. Pick any point $p \in X$ and let $\{x_\alpha\}_{\alpha}$ be a basis of $\X_{p}$ and $\{x_\alpha^*\}_{\alpha}$ be its dual. Then, $d'$ can be written as follows:
\begin{align}
d'(m \otimes n) = (-1)^{|m|}\sum_{\alpha} mx_{\alpha} \otimes x_{\alpha}^* n.
\end{align} 
Here, note that the sum is a finite sum. In fact, since $\M^i=0$ for $i < N$, we have $mx_{\alpha}=0$ for $|x_{\alpha}| < N-|m|$. As in the case of finite length, we can check that this defines a differential by direct calculations. More precisely, the following holds:
\begin{lem} \label{lem:diff}
In the situation above, we have the following formulae.
\begin{enumerate}
\item $d'^2=0$. 
\item $d_{\M \otimes_{\str} \N} \circ d' + d' \circ d_{\M \otimes_{\str} \N}=0$.
\end{enumerate}
Consequently, we have $(d_{\M \otimes_{\str} \N}+d')^2=0$.
\end{lem}
\begin{proof}
Pick any point $p \in X$. It suffices to show $(1)$ and $(2)$ on the stalk at $p$ and let $\{x_{\alpha}\}_{\alpha \in A}$ be a basis of $\X_p$. Note that the following calucualtions are all finite sums due to the boundedness condition on $\M$.\\
(1) For a section $m \otimes n \in \M \otimes_{\str} \N$, we have
\begin{align*}
d'^2(m \otimes n) &= d'\left( (-1)^{|m|} \sum_{\alpha \in A} mx_{\alpha} \otimes x_{\alpha}^*n \right)\\
&=(-1)^{|m|} \sum_{\alpha \in A} d'( mx_{\alpha} \otimes x_{\alpha}^*n) \\
&=(-1)^{|m|} \sum_{\alpha \in A} (-1)^{|mx_{\alpha}|} \left( \sum_{\beta \in A} mx_{\alpha}x_{\beta} \otimes x_{\beta}^* x_{\alpha}^* n \right)\\
&=\sum_{\alpha,\beta \in A} (-1)^{|x_{\alpha}|}mx_{\alpha}x_{\beta} \otimes x_{\beta}^* x_{\alpha}^* n.
\end{align*}
This is written as
\begin{align*}
\sum_{\alpha,\beta \in A} (-1)^{|x_{\alpha}|}mx_{\alpha}x_{\beta} \otimes x_{\beta}^* x_{\alpha}^* n=&\sum_{\alpha \in A} (-1)^{|x_{\alpha}|}mx_{\alpha}x_{\alpha} \otimes x_{\alpha}^* x_{\alpha}^* n + \sum_{\substack{\alpha,\beta \in A \\ \alpha \neq \beta}} (-1)^{|x_{\alpha}|}mx_{\alpha}x_{\beta} \otimes x_{\beta}^* x_{\alpha}^* n.
\end{align*}
Consider the first term. By the definition of $\A$ and $\B$, if $|x_{\alpha}|$ is odd, then $|x_{\alpha}^*|$ is even and vice versa. Hence we always have $x_{\alpha}^2=0$ or $(x_{\alpha}^*)^2=0$. Thus we obtain
\begin{align*}
\sum_{\alpha \in A} (-1)^{|x_{\alpha}|}mx_{\alpha}x_{\alpha} \otimes x_{\alpha}^* x_{\alpha}^* n = 0.
\end{align*}
Consider the second term. First, this can be written as
\begin{align*}
&\sum_{\substack{\alpha,\beta \in A \\ \alpha \neq \beta}} (-1)^{|x_{\alpha}|}mx_{\alpha}x_{\beta} \otimes x_{\beta}^* x_{\alpha}^* n \\
&= \sum_{\substack{\{ \alpha,\beta \} \subset A \\ \alpha \neq \beta}} \left( (-1)^{|x_{\alpha}|}mx_{\alpha}x_{\beta} \otimes x_{\beta}^* x_{\alpha}^* n + (-1)^{|x_{\beta}|}mx_{\beta}x_{\alpha} \otimes x_{\alpha}^* x_{\beta}^* n \right)\\
&= \sum_{\substack{ \{ \alpha,\beta \} \subset A \\ \alpha \neq \beta}} ((-1)^{|x_{\alpha}|}+(-1)^{|x_{\beta}|+|x_{\alpha}||x_{\beta}|+|x_{\beta}^*||x_{\alpha}^*|})mx_{\alpha}x_{\beta} \otimes x_{\beta}^* x_{\alpha}^* n.
\end{align*}
Here, we used the graded commutative relation $x_{\alpha} x_{\beta} = (-1)^{|x_{\alpha}||x_{\beta}|}x_{\beta} x_{\alpha}$ and $x_{\alpha}^* x_{\beta}^* = (-1)^{|x_{\alpha}^*||x_{\beta}^*|}x_{\beta} ^*x_{\alpha}^*$. Then, as noticed above, we have
\begin{align*}
\ \ \ \ \ \ \ \ \ \ |x_{\beta}^*||x_{\alpha}^*| &\equiv (|x_{\beta}|-1)(|x_{\alpha}|-1) \equiv |x_{\beta}||x_{\alpha}| -|x_{\beta}| - |x_{\alpha}| +1 \pmod 2 .
\end{align*}
From this, we obtain
\begin{align*}
\sum_{\substack{\alpha,\beta \in A \\ \alpha \neq \beta}} (-1)^{|x_{\alpha}|}mx_{\alpha}x_{\beta} \otimes x_{\beta}^* x_{\alpha}^* n &=  \sum_{\substack{\{ \alpha,\beta \} \subset A \\ \alpha \neq \beta}} ((-1)^{|x_{\alpha}|}+(-1)^{|x_{\alpha}|+1})mx_{\alpha}x_{\beta} \otimes x_{\beta}^* x_{\alpha}^* n \\
 &= 0.
\end{align*}
This proves $d'^2=0$. \\
(2) First we calculate $d_{\M \otimes_{\str} \N} \circ d' (m\otimes n)$ and $d' \circ d_{\M \otimes_{\str} \N} (m \otimes n)$ for a section $m \otimes n \in \M \otimes_{\str} \N$:
\begin{align*}
&d_{\M \otimes_{\str} \N} \circ d' (m\otimes n) \\
&= d_{\M \otimes_{\str} \N} \left( (-1)^{|m|}\sum_{\alpha \in A} mx_{\alpha} \otimes x_{\alpha}^*n \right) \\
&= (-1)^{|m|}\sum_{\alpha \in A} d_{\M \otimes_{\str} \N} (mx_{\alpha} \otimes x_{\alpha}^*n) \\
&= (-1)^{|m|}\sum_{\alpha \in A} \left( d_{\M} (mx_{\alpha}) \otimes x_{\alpha}^* n + (-1)^{|mx_{\alpha}|} mx_{\alpha} \otimes d_{\N}(x_{\alpha}^* n) \right) \\
&=  (-1)^{|m|}\sum_{\alpha \in A} \left( d_{\M}(m)x_{\alpha} \otimes x_{\alpha}^* n +(-1)^{|m|}md_{\B}(x_{\alpha}) \otimes x_{\alpha}^* n \right) \\
&\ \ \ \ + (-1)^{|x_{\alpha}|} \sum_{\alpha \in A} \left( mx_{\alpha} \otimes d_{\A}(x_{\alpha}^*)n + (-1)^{|x_{\alpha}^*|} mx_{\alpha} \otimes x_{\alpha}^*d_{\N}(n) \right),\\
&d' \circ d_{\M \otimes_{\str} \N} (m \otimes n)  \\
&= d'(d_{\M}(m) \otimes n + (-1)^{|m|}m \otimes d_{\N}(n))\\
&=(-1)^{|m|+1} \sum_{\alpha \in A} d_{\M}(m)x_{\alpha} \otimes x_{\alpha}^*n + \sum_{\alpha \in A} mx_{\alpha} \otimes x_{\alpha}^*d_{\N}(n).
\end{align*}
Since $|x_{\alpha}|+|x_{\alpha}^*|$ is always odd, we obtain
\begin{align*}
&(d_{\M \otimes_{\str} \N} \circ d' + d' \circ d_{\M \otimes_{\str} \N} )(m \otimes n) \\
&= \sum_{\alpha \in A} \left( md_{\B}(x_{\alpha}) \otimes x_{\alpha}^* n + (-1)^{|x_{\alpha}|} mx_{\alpha} \otimes d_{\A}(x_{\alpha}^*)n \right).
\end{align*}
Let us show that this is indeed zero. To see this, we will write $d_{\B}(x_{\alpha})$ and $d_{\A}(x_{\alpha}^*)$ in more explicit way. Let $\{ x_1^{(N)}, \cdots,  x_{\ell_N}^{(N)} \}$ be a basis of $\E_{-N,p}$ and $\{ (x_1^{(N)})^*, \cdots,  (x_{\ell_N}^{(N)})^* \}$ be its dual. We write
\begin{align*}
d_{\B}(x_i^{(N)}) =\sum_{j=1}^{\ell_{N-1}} a_{ij}^{(N)}x_{j}^{(N-1)},
\end{align*}
for some $a_{ij}^{(N)} \in \mathcal{O}_{X,p}$. Then we have
\begin{align*}
d_{\A}((x_i^{(N)})^*) =-(-1)^n\sum_{j=1}^{\ell_{N+1}} a_{ji}^{(N+1)}(x_{j}^{(N+1)})^*.
\end{align*}
Now, by definition, we have $\{x_{\alpha}\} = \{x_{i}^{(N)}\}$. Hence we have
\begin{align*}
&\sum_{\alpha \in A} \left( md_{\B}(x_{\alpha}) \otimes x_{\alpha}^* n + (-1)^{|x_{\alpha}|} mx_{\alpha} \otimes d_{\A}(x_{\alpha}^*)n \right) \\
&=\sum_{i,N}\left( \sum_{j=1}^{\ell_{N-1}} ma_{ij}^{(N)}x_{j}^{(N-1)} \otimes (x_{i}^{(N)})^* n - \sum_{j=1}^{\ell_{N+1}} mx_{i}^{(N)} \otimes a_{ji}^{(N+1)}(x_{j}^{(N+1)})^*n \right) \\
&=\sum_{i,N}\left( \sum_{j=1}^{\ell_{N-1}} ma_{ij}^{(N)}x_{j}^{(N-1)} \otimes (x_{i}^{(N)})^* n - \sum_{j=1}^{\ell_{N}} mx_{i}^{(N-1)} \otimes a_{ji}^{(N)}(x_{j}^{(N)})^*n \right)\\
&=\sum_{i,N}\left( \sum_{j=1}^{\ell_{N-1}} a_{ij}^{(N)}( mx_{j}^{(N-1)} \otimes (x_{i}^{(N)})^* n) - \sum_{j=1}^{\ell_{N}} a_{ji}^{(N)}( mx_{i}^{(N-1)} \otimes (x_{j}^{(N)})^*n) \right)\\
&=\sum_{N}\left( \sum_{i,j} a_{ij}^{(N)}( mx_{j}^{(N-1)} \otimes (x_{i}^{(N)})^* n) - \sum_{i,j} a_{ji}^{(N)}( mx_{i}^{(N-1)} \otimes (x_{j}^{(N)})^*n) \right) \\
&=0.
\end{align*}
This proves $d_{\M \otimes_{\str} \N} \circ d' + d' \circ d_{\M \otimes_{\str} \N}=0$.
\end{proof}

When $\N \in C^-(\A)$, $\M \otimes^K \N$ is defined similarly. Assume $\N^i =0$ for $i >N'$. Then we set $\M \otimes^K \N = \M \otimes_{\str} \N$ as a graded $\str$-module, and its differential is $d_{\M \otimes_{\str} \N}+d''$, where $d''$ is the Koszul differential defined as follows. For a section $m\otimes n$ of $\M \otimes_{\str} \mathcal{N}$, $d''(m \otimes n)$ is defined to be the following compositions
\begin{align*}
\M \otimes_{\str} \mathcal{N} &\to \M \otimes_{\str} \mathcal{N} \to \M \otimes_{\str} \X^{(N'-|n|)} \otimes_{\str} (\X^{(N'-|n|)})^{\vee} \otimes_{\str} \mathcal{N} \to \M \otimes_{\str} \mathcal{N},
\end{align*}
where the morphisms are the same ones as defined above. Then, $d''$ can be written as
\begin{align}
d''(m \otimes n) = (-1)^{|m|}\sum_{\alpha} mx_{\alpha} \otimes x_{\alpha}^* n
\end{align} 
at each stalk. Here, note that this is a finite sum since $x_{\alpha}^* n$ = 0 if $|x_{\alpha}^*| > N'-|n|$ in this case. Hence $d''$ is defined by the same formula as (3.1), and it follows that $(d_{\M \otimes_{\str} \mathcal{N}}+d'')^2=0$.

An $(\A, \B)$-bimodule $\N \otimes^K \M$ is defined similarly, with the sign modification as in \S 3.1, under the condition $\M \in C^+(\B)$ or $\N \in C^-(\A)$. Then, as in the case of finite length, we define $\K_{\A}$ and $\K_{\B}$ as follows:
\begin{align*}
\K_{\A} := \A \otimes^K \B^{\vee},\\
\K_{\B} := \B \otimes^K \A^{\vee},
\end{align*}
where duals are taken in the graded sense. As in the case of finite length, there are a homomorphism of dg-$\A$-modules $\K_{\A} \to \str$ and a homomorphism of dg-$\B$-modules $\str \to \K_{\B}$. We will show that these homomorphisms give K-projective resolutions of $\str$ at each stalk.
\begin{prop}
\ 

\begin{enumerate}
\item The homomorphism of dg-$\A$-modules $\K_{\A} \to \str$ is a quasi-isomorphism.
\item The homomorphism of dg-$\B$-modules $\str \to \K_{\B}$ is a quasi-isomorphism.
\end{enumerate}
\end{prop}
\begin{proof}
First we define
\begin{align*}
\B^{(n)} &:= \Sym \X^{(n)} = \Sym (0 \to \E^{-n} \to \cdots \to \E^0 \to 0), \\
\A^{(n)} &:= \Sym ((\X^{(n)})^{\vee}[-1]) =\Sym (0 \to (\E^{0})^{\vee} \to \cdots \to (\E^{-n})^{\vee} \to 0).
\end{align*}
Then we have 
\begin{align*}
(\K_{\A})_{\leq n} &= \left(\A^{(n)} \otimes^K (\B^{(n)})^{\vee} \right)_{\leq n}, \\
(\K_{\B})_{\geq -n} &= \left(\B^{(n)} \otimes^K (\A^{(n)})^{\vee} \right)_{\geq -n}.
\end{align*}
Hence we obtain, for any $p < n$, 
\begin{align*}
H^p(\K_{\A}) &= H^p((\K_{\A})_{\leq n})=H^p((\A^{(n)} \otimes^K (\B^{(n)})^{\vee} )_{\leq n})=H^p(\A^{(n)} \otimes^K (\B^{(n)})^{\vee}),\\
H^{-p}(\K_{\B}) &= H^p((\K_{\B})_{\geq -n})=H^{-p}((\B^{(n)} \otimes^K (\A^{(n)})^{\vee})_{\geq -n})=H^{-p}(\B^{(n)} \otimes^K (\A^{(n)})^{\vee}).
\end{align*}
Since $n$ is arbitrary, by Proposition \ref{qis1}, we get the desired result.
\end{proof}

\begin{prop}
Let $p \in X$ be any point. Then we have
\begin{enumerate}
\item $\K_{\A,p}$ is K-projective dg-module over $\A_p$.
\item $\K_{\B,p}$ is K-projective dg-module over $\B_p$.
\end{enumerate}
\end{prop}
\begin{proof}
For (2), the same proof as Proposition \ref{prop:proj1} (2) can be applied, so we will concentrate on (1). Note that the proof for (1) in Proposition \ref{prop:proj1} cannot be applied to this case since the index set $I_{\ell}$ appearing in the proof of Proposition \ref{prop:proj1} is infinite.

In this proof, we omit the symbol ``$p$'' for simplicity and as in the proof of Proposition \ref{prop:proj1}, we define
\begin{align*}
\K_{\B}(i_0,\cdots,i_n) := \B \otimes_{\str} S^{i_0}((\E^{0})^{\vee})^{\vee} \otimes_{\str} \cdots \otimes_{\str} S^{i_n}((\E^{-n})^{\vee})^{\vee},
\end{align*}
where $S^i((\E^j)^{\vee})$ is defined by
\begin{align*}
S^i((\E^j)^{\vee}) :=
\begin{cases}
\Sym^i ((\E^j)^{\vee})\ &\text{if $j$ is odd},\\
\wedge^i (\E^j)^{\vee}\ &\text{if $j$ is even}.
\end{cases}
\end{align*}
for $i \in \mathbb{Z}_{\geq 0}$ and $j \in \mathbb{Z}_{\leq 0}$. Then, $\K_{B}$ is decomposed to a direct sum
\begin{align*}
\K_{\B} = \bigoplus \K_{\B}(i_0,\cdots,i_n)
\end{align*}
as an $\str$-module, where the direct sum is taken over indices $(i_0, \cdots ,i_n,\cdots) \in \mathbb{Z}_{\geq 0}^{\oplus \mathbb{N}}$ with $i_n = 0$ for all but finitely many $n$. 

To prove the K-projectivity by using Lemma \ref{lem.spa}, we will construct a filtration $\{F_{\lambda}\}$ of $\K_{\B}$ inductively. First, for $n \in \mathbb{Z}_{\geq 0}$ we define a dg-submodule $\K_n$ of $\K_{\B}$ by
\begin{align*}
\K_n := \bigoplus_{i_0 + \cdots + i_n\leq n} \K_{\B}(i_0, \cdots , i_n).
\end{align*}
By the description of differentials in Proposition \ref{prop:proj1}, this is indeed a dg-submodule of $\K_{\B}$. We assume that we can pick a finite filtration of $\K_n$ $\{F_{\lambda}\}_{\lambda=-1}^{m}$ consisting of dg-submodules of $\K_{\B}$ which satisfies the following two conditions:
\begin{itemize}
\item $F_{-1}=0$ and $\bigcup_{\lambda=-1}^{m}F_{\lambda} = \K_n$,
\item Each subquotient $F_{\lambda}/F_{\lambda-1}$ is isomorphic to a tensor product of $\B$ and a free $\str$-module.
\end{itemize}
Then, we will construct a finite filtration of $\K_{n+1}$ which satisfies the same conditions as above.

Let $I_n :=\{(i_0,\cdots, i_n) | i_0+\cdots +i_n \leq n \}$ be the set of indices appearing in $\K_n$. We define a total order on $I_{n+1} \setminus I_n$ by the same way as in the proof Proposition \ref{prop:proj1}. More precisely, for $\bm{i}=(i_0,\cdots,i_{n+1}) \in I_{n+1} \setminus I_n$ and $\bm{j}=(j_0,\cdots,j_{n+1}) \in I_{n+1} \setminus I_n$,
\begin{align*}
\bm{i} \geq \bm{j} \overset{\text{def.}}{\Longleftrightarrow}
\begin{cases}
\sum_{s=0}^{n+1}i_s>\sum_{s=0}^{n+1}j_s,\ \text{or} \\ 
\sum_{s=0}^{n+1}i_s=\sum_{s=0}^{n+1}j_s \text{ and } \bm{i} \leq \bm{j} \ \text{in the lexicographic order}.
\end{cases}
\end{align*}
Then we write $I_{n+1} \setminus I_n =\{\bm{i}_1, \bm{i}_2, \cdots , \bm{i}_\ell \}$ so that $\bm{i}_s < \bm{i}_t$ if $s<t$. We set $F_{m+s} := F_m \oplus \bigoplus_{j=0}^s \B(\bm{i}_j)$ for $s=1,\cdots, \ell$. Then by the descripiton of differentials in Proposition \ref{prop:proj1}, each $F_{m+j}$ is a dg-submodule of $\K_{\B}$ and hence $\{F_\lambda \}_{\lambda=-1}^{m+\ell}$ gives a desired filtration of $\K_{n+1}$. Repeating this process, we can obtain a filtration of $\K_{\B}$ which satisfies the conditions in Lemma \ref{lem.spa}. Hence $\K_{\B}$ is a K-projective dg-module over $\B$.
\end{proof}

\subsection{Proof of the Equivalences}
In this section, we construct the functors between $\D(\A)$ and $\D(\B)$, and prove that this induces equivalences $\D(\A) \cong \operatorname{IndCoh}_X\B$ and $D^+(\A) \cong D^+(\B)$. Here, $\operatorname{IndCoh}_X\B$ denotes the Ind-completion of the thick subcategory of $\D(\B)$ generated by $\str$. From now on, we fix a connective complex (of possibly infinite length) of locally free sheaves of finite rank
\begin{align*}
\X = (\cdots \to \E^{-n} \to \cdots \to \E^{-1} \to \E^{0} \to 0),
\end{align*}
where sections of $\E^{-n}$ have cohomological degree $-n$, and we set $\B :=\Sym \X$. We define the Koszul dual $\A$ of $\B$ to be $\A := \Sym (\X^{\vee}[-1])$, and we give an internal grading on $\A$ and $\B$ by considering sections of $\E^i$ and $(\E^i)^{\vee}$ to have internal degree $1$, $-1$, respectively. 

Our goal in this section is to prove the Koszul duality equivalences between $\A$ and $\B$. Before stating this result, we prepare some notations.
\begin{dfn}\label{def:cohx}
Let $X$ be a affine Noetherian scheme over $k$. Let $\operatorname{Coh}_X \B$ (resp. $\operatorname{Coh}_{gr,X} \B$) denote the thick subcategory of $\D(\B)$ (resp. $\D_{gr}(\B)$) generated by $\str$, and let $\operatorname{IndCoh}_X \B$ (resp. $\operatorname{IndCoh}_{gr,X} \B$) be its Ind-completion.

For a general Noetherian scheme $X$ over $k$, we define $\operatorname{IndCoh}_X \B$ to be the following limit 
\begin{align*}
\operatorname{IndCoh}_X \B := \lim_{\substack{U \subset X \\ U:\text{affine open}}} \operatorname{IndCoh}_U \B|_U,
\end{align*}
and we define $\operatorname{IndCoh}_{gr,X} \B$ similarly.
\end{dfn}
Then, the Koszul duality in this paper can be formulated as follows.
\begin{thm}\label{thm:main}
There is an equivalence of $\infty$-categories
\begin{align*}
\D_{gr}(\A) \cong \operatorname{IndCoh}_{gr,X} \B,
\end{align*}
and if we restrict this equivalence to the bounded-below subcategories, we obtain an equivalence of triangulated categories
\begin{align*}
D_{gr}^+(\A) \cong D_{gr}^+ (\B).
\end{align*} 
If moreover $\E_0=0$, we also have the following equivelences in the ungraded context:
\begin{align*}
\D(\A) &\cong \operatorname{IndCoh}_X \B,\\
D^+(\A) &\cong D^+ (\B).
\end{align*}
\end{thm}
First we explain the reason for the latter condition $\E_0=0$. If $\E^0 \neq 0$, $\B^p$ is not coherent over $\str$ since $\B^0=\Sym \E^0$, for example. Due to this, we cannot obtain the isomorphism in Lemma \ref{lem:fun}, and the proof does not work.

However, this problem can be ignored when the additional grading on $\X$ is considered. In this case, all $\B^p_{q}$ are coherent over $\str$, and hence all calculations can be done by simply replacing $\Hom$'s by graded one, so we will give a proof only when $\A$ and $\B$ are ungraded.

To prove our equivalences, first we define two functors $F$ and $G$ between $C(\A)$ and $C(\B)$ as follows:
\begin{align*}
F:C(\A) \to C(\B), \ &\M \mapsto \K_{\A} \otimes_{\A} \M ,\\
G:C(\B) \to C(\A), \ &\M \mapsto \Hom_{\B}(\K_{\B} , \M).
\end{align*}
We will show that $F$ induces equivalence in Theorem \ref{thm:main}. To see  this, first we will show that these functors give adjouint functors between $C^+(\A)$ and $C^+(\B)$.

\begin{lem} \label{lem:F}
For any $\M \in C^{+}(\A)$, $F(\M) \in C^{+}(\B)$. Consequently, the restriction of $F$ gives a functor $C^+(\A) \to C^+(\B)$.
\end{lem}
\begin{proof}
First, for any $\M \in C^+({\A})$, we have isomorphisms
\begin{align*}
F(\M) &= \K_{\A} \otimes_{A} \M \\
&= (\A \otimes^K \B^{\vee} ) \otimes_{\A} \M \\
&\cong \B^{\vee} \otimes^K \M.
\end{align*}
Since $\B^{\vee} \in C^+(\B)$, this shows that $F(\M) \in C^{+}(\B)$. Hence the restriction of $F$ defines a functor $C^{+}(\A) \to C^{+}(\B)$.
\end{proof}

The functor $G$ can be described as follows.

\begin{lem} \label{lem:fun}
For $\M \in C^{+}(\B)$, there is a functorial isomorphism
\begin{align*}
G(\M) \cong \A \otimes^K \M.
\end{align*}
In particular, $G(\M) \in C^+(\A)$ and hence the restriction of $G$ gives a functor $C^+(\B) \to C^+(\A)$.
\end{lem}
\begin{proof}
First, as graded $\str$-modules (that is, forgetting differentials), there are isomorphisms
\begin{align*}
G(\M) = \Hom_{\B}(\K_{\B},\M) = \Hom_{\B}(\B \otimes_{\str} \A^\vee,\M) \cong \Hom_{\str} (\A^\vee,\M) 
\end{align*}
Now, let us prove that the natural morphism $\A \otimes_{\str} \M \to \Hom_{\str} (\A^\vee,\M)$ gives an isomorphism. The $n$-th degree of $\Hom_{\str} (\A^\vee,\M)$ is
\begin{align*}
\Hom_{\str}^n (\A^\vee,\M) &= \prod_{q-p=n} \Hom_{\str} ((\A^{\vee})^p,\M^q) = \bigoplus_{q-p=n} \Hom_{\str} ((\A^{-p})^{\vee},\M^q).%\\
\end{align*}
Here, since $\A^{-p} =0$ for $p \ll 0$ and $\M^q =0$ for $q \ll 0$, the direct product $\prod_{q-p=n} \Hom_{\str} ((\A^{\vee})^p,\M^q)$ is finite, and hence it is the same thing as a direct sum. Therefore the $n$-th degree of $\Hom_{\str} (\A^\vee,\M)$ coincides with that of $\A \otimes_{\str} \M$, and the natural morphism induces isomorophisms at each degree. Thus we obtain the isomorphism of graded $\str$-modules $\Hom_{\B}(\K_{\B},\M) \cong \A \otimes^K \M$. Moreover, this gives an isomorphism as dg-$\A$-modules. It is obvious that this isomorphism preserves $\A$-action, so we will show that $\varphi : \A \otimes^K \M \to \Hom_{\B}(\K_{\B}, \M)$ commutes with the differentials. To see this, for a section $a \otimes m$ of $\A \otimes^K \M$, we should check that
\begin{align*}
\varphi ((d_{\A \otimes_{\str} \M} + d')(a \otimes m)) = d_{\Hom_{\B}(\K_{\B},\M)}(\varphi (a \otimes m)),
\end{align*}
where $d'$ is the Koszul differential. Now, the right hand side is
\begin{align*}
&d_{\Hom_{\B}(\K_{\B},\M)}(\varphi (a \otimes m)) \\
&= d_{\M} \circ \varphi (a \otimes m) -(-1)^{|\varphi (a \otimes m)|}\varphi (a \otimes m) \circ d_{\K_{\B}}\\
&= d_{\M} \circ \varphi (a \otimes m) -(-1)^{|a|+|m|} \varphi (a \otimes m) \circ (d_{\B \otimes_{\str} \A^{\vee}}+d''),
\end{align*}
where $d''$ is the Koszul differential of $\K_{\B}$. Since we have
\begin{align*}
\varphi(d_{\A \otimes_{\str} \M}(a\otimes m)) = d_{\M} \circ \varphi (a \otimes m) -(-1)^{|a|+|m|} \varphi (a \otimes m) \circ d_{\B \otimes_{\str} \A^{\vee}},
\end{align*}
it is enough to show that $\varphi(d'(a\otimes m)) = -(-1)^{|a|+|m|}\varphi (a \otimes m) \circ d''$. Then we can check this by direct calculations, using the fact that $\varphi$ is given by
$\varphi(a\otimes m) (b \otimes f) =(-1)^{|a|+|a||m|+|b||m|+|b||f|}f(1)bm$ for a local section $b \otimes f$ of $\K_{\B}$ and the explicit formula of Koszul differentials (3.2). In fact, let $\{x_{\alpha}\}$ be a basis of $\X_p$ for a point $p\in X$, then we have
	\begin{align*}
	&\varphi(d'(a\otimes m))(b \otimes f) \\
	&= \varphi \left( -(-1)^{|a|}\sum_{\alpha} ax_{\alpha}^* \otimes x_{\alpha}m \right) (b\otimes f)\\
	&=-(-1)^{|a|}\sum_{\alpha} (-1)^{|ax_{\alpha}^*|+|ax_{\alpha}^*||x_{\alpha}m|+|b||x_{\alpha}m|+|b||f|} b f(ax_{\alpha}^*) x_\alpha m,
	\end{align*}
	and
	\begin{align*}
	&-(-1)^{|a|+|m|} \varphi(a \otimes m) \circ d''(b \otimes f) \\
	&=-(-1)^{|a|+|m|} \varphi(a\otimes m) \left( (-1)^{|b|}\sum_{\alpha} bx_{\alpha} \otimes x_{\alpha}^*f \right) \\
	&=-(-1)^{|a|+|m|}(-1)^{|b|} \sum_{\alpha}  (-1)^{|a|+|a||m|+|bx_{\alpha}||m|+|bx_{\alpha}||x_{\alpha}^*f|} bx_{\alpha} x_{\alpha}^*f(a)m\\
	&=-(-1)^{|a|+|m|+|b|} \sum_{\alpha}  (-1)^{|a|+|a||m|+|bx_{\alpha}||m|+|bx_{\alpha}||x_{\alpha}^*f|+|x_{\alpha}^*||f|}bx_{\alpha} f(x_{\alpha}^*a)m\\
	&=-(-1)^{|a|+|m|+|b|} \sum_{\alpha}  (-1)^{|a|+|a||m|+|bx_{\alpha}||m|+|bx_{\alpha}||x_{\alpha}^*f|+|x_{\alpha}^*||f|+|x_{\alpha}^*||a|}bx_{\alpha} f(ax_{\alpha}^*)m.
	\end{align*}
	Hence it is enough to compare the signs of each term. Here, note that $f(ax_{\alpha}^*)$ is zero unless $|f|=-|ax_{\alpha}^*|$, so we may assume that $|f|=-|ax_{\alpha}^*|$. Let us calculate the exponents of $-1$ in both hands. The following caluculation is always done in modulo $2$ and notice that $|x_{\alpha}|$ and $|x_{\alpha}^*|$ always have different parity; we frequently use equations such as $|x_{\alpha}||x_{\alpha}^*| \equiv 0$, $|x_{\alpha}|+1 \equiv |x_{\alpha}^*|$, and $|x_{\alpha}|+|x_{\alpha}^*| \equiv 0$. First, the exponent in the left hand side is,
	\begin{align*}
	&1+|a| + |ax_{\alpha}^*|+|ax_{\alpha}^*||x_{\alpha}m| +|b||x_{\alpha}m| +|b||f|\\
	&\equiv 1+ |a| + |a| + |x_{\alpha}^*| + (|a|+|x_{\alpha}^*|)(|x_{\alpha}|+|m|) + |b||x_{\alpha}m| + |b||f| \\
	&\equiv 1+|x_{\alpha}^*| + |a||x_{\alpha}| + |a||m| + |x_{\alpha}^*||m| +|b||x_{\alpha}|+|b||m|+ |b||f|.
	\end{align*}
	One in the right hand side is
	\begin{align*}
	&1+|a| + |m| + |b| +|a|+|a||m|+|bx_{\alpha}||m| +|bx_{\alpha}||x_{\alpha}^*f|+|x_{\alpha}^*||f| +|x_{\alpha}^*||a|\\
	&\equiv 1+|m|+|b|+|a||m| +|b||m|+|x_{\alpha}||m|  \\
	&\ \ \ \ \ \  \  \ \ \ \ \ \  \ \ \ \ \ \ \ \ \ \ \ +(|b|+|x_{\alpha}|)(|x_{\alpha}^*|+|f|)+ |x_{\alpha}^*||f| +|x_{\alpha}^*||a|\\
	&\equiv 1+|m|+|b|+|a||m| +|b||m|+|x_{\alpha}||m| \\
	&\ \ \ \ \ \  \  \ \ \ \ \ \  \ \ \ \ \ \ \ \ \ \ \ +|b||x_{\alpha}^*|+|x_{\alpha}||f|+|b||f|+ |x_{\alpha}^*||f| +|x_{\alpha}^*||a|\\
	&\equiv 1+(1+|x_{\alpha}|)|m|+(1+|x_{\alpha}^*|)|b|+|a||m| \\
	&\ \ \ \ \ \  \  \ \ \ \ \ \  \ \ \ \ \ \ \ \ \ \ \ +|b||m|+|b||f|+(|x_{\alpha}|+|x_{\alpha}^*|)|f| + |x_{\alpha}^*||a|\\
	&\equiv 1+|x_{\alpha}^*||m|+|b||x_{\alpha}|+|a||m|+|b||m|+|b||f|+|f|+|x_{\alpha}^*||a|\\
	&\equiv 1+|x_{\alpha}^*||m|+|b||x_{\alpha}|+|a||m|+|b||m|+|b||f|+|ax_{\alpha}^*|+|x_{\alpha}^*||a|\\
	&\equiv 1+|x_{\alpha}^*||m|+|b||x_{\alpha}|+|a||m|+|b||m|+|b||f|+|a|+|x_{\alpha}^*|+|x_{\alpha}^*||a|\\
	&\equiv 1+|x_{\alpha}^*||m|+|b||x_{\alpha}|+|a||m|+|b||m|+|b||f|+(1+|x_{\alpha}^*|)|a|\\
	&\equiv 1+|x_{\alpha}^*|+|x_{\alpha}^*||m|+|b||x_{\alpha}|+|a||m|+|b||m|+|b||f|+|a||x_{\alpha}|.
	%&=1+|m|+|b|+|a||m| + |b||x_{\alpha}^*| + |b||f| + |x_{\alpha}||f| + |x_{\alpha}^*||f| + |x_{\alpha}^*||a|
	\end{align*}
Hence we have $\varphi(d'(a\otimes m)) = -(-1)^{|a|+|m|}\varphi (a \otimes m) \circ d''$, and the claim follows.
\end{proof}

\begin{lem}\label{lem:adj}
The functor $G:C^+(\B) \to C^+(\A)$ is left adjoint to $F:C^+(\A) \to C^+(\B)$. Namely, for any $\M \in C^+(\A)$ and $\N \in C^+(\B)$, there is a functorial isomorphism
\begin{align*}
\operatorname{Hom}_{\A} (G(\M) , \N) \cong \operatorname{Hom}_{\B} (\M , F(\N)).
\end{align*}
\end{lem}
\begin{proof}
First, forgetting differentials, we have the isomorphisms
\begin{align*}
\operatorname{Hom}_{\operatorname{Gr-}\A} (G(\M) , \N) &= \operatorname{Hom}_{\operatorname{Gr-}\A}(\A \otimes_{\str} \M, \N) \\
& \cong \operatorname{Hom}_{\operatorname{Gr-}\str} (\M,\N), \\
\operatorname{Hom}_{\operatorname{Gr-}\B} (\M , F(\N)) &= \operatorname{Hom}_{\operatorname{Gr-}\B}(\M ,\B^{\vee}\otimes_{\str} \N) \\
&\cong \operatorname{Hom}_{\operatorname{Gr-}\B}(\M ,\Hom_{\str}(\B,\N))\\
&\cong \operatorname{Hom}_{\operatorname{Gr-}\str}(\M,\N).
\end{align*}
Here, note that we need the fact that $\B \in C^-(\B)$ and $\N \in C^+(\A)$ to deduce that $\B^{\vee}\otimes_{\str} \N \cong \Hom_{\str}(\B,\N).$ By the calculation above, both $\operatorname{Hom}_{\A} (G(\M) , \N)$ and $\operatorname{Hom}_{\B} (\M , F(\N))$ can be considered as subsets of $\operatorname{Hom}_{\operatorname{Gr-}\str}(\M,\N)$. Then what we need to prove is these subsets are indeed the same.

Let $f \in \operatorname{Hom}_{\operatorname{Gr-}\str}(M,N)$, and let $\tilde{f} \in \operatorname{Hom}_{\operatorname{Gr-}\B} (\M , F(\N))$ and $\hat{f} \in \operatorname{Hom}_{\operatorname{Gr-}\A} (G(\M) , \N)$ be corresponding morphisms. Then we need to show that 
\begin{align*}
\tilde{f} \in \operatorname{Hom}_{\B} (\M , F(\N)) \Longleftrightarrow \hat{f} \in \operatorname{Hom}_{\A} (G(\M) , \N).
\end{align*}
More explicitly, this condition means $\tilde{f}$ satisfies 
\begin{align*}
d_{F(\N)} \circ \tilde{f} = \tilde{f} \circ d_{\M}
\end{align*} 
if and only if $\hat{f}$ satisfies 
\begin{align*}
d_{\N} \circ \hat{f} = \hat{f} \circ d_{G(\M)}.
\end{align*}
However, this condition can be checked locally, so we may consider each stalk. Let $p \in X$ and $\{x_{\alpha}\}$ be a basis of $\X_p$. First consider $F(\N)$. There is an isomorophism as a graded $\str$-module
\begin{align*}
F(\N) = \B^{\vee} \otimes_{\str} \N \cong \Hom_{\str}(\B,\N).
\end{align*}
Then, a differential on $\Hom_{\str}(\B,\N)$ induced by $\B^{\vee} \otimes_{\str} \N$ is $d_{\Hom_{\str}(\B,\N)} +d'$, where $d'$ is defined by
\begin{align*}
d'(\varphi)(b) = (-1)^{|b|} \sum_{\alpha} x_{\alpha}^* \varphi(x_{\alpha}b),
\end{align*}
for $\varphi \in \Hom_{\str}(\B,\N).$ Hence, through the isomorphism above, what we need to show is $\tilde{f}$ satisfies 
\begin{align}
(d_{\Hom_{\str}(\B,\N)}+d') \circ \tilde{f} = \tilde{f} \circ d_{\M},
\end{align} 
if and only if $\hat{f}$ satisfies 
\begin{align}
d_{\N} \circ \hat{f} = \hat{f} \circ d_{G(\M)}.
\end{align}

Now, $\tilde{f}$ is given by $\tilde{f}(m)(b)=f(bm)$. Since $\tilde{f}$ is a homomorphism of graded $\B$-modules, (3.3) is equivalent to $((d_{\Hom_{\str}(\B,\N)}+d') \circ \tilde{f}(m))(1) = (\tilde{f} \circ d_{\M}(m))(1)$ for all $m \in \M$. Then we have
\begin{align*}
(\tilde{f}\circ d_{\M}(m))(1) = f(d_{\M}(m)),
\end{align*}
and
\begin{align*}
((d_{\Hom_{\str}(\B,\N)}+d') \circ \tilde{f}(m))(1) = d_{\N} (f(m)) +\sum_{\alpha}x_{\alpha}^*f(x_{\alpha}m).
\end{align*}
Hence (3.3) is equivalent to
\begin{align}
f(d_{\M}(m))=d_{\N} (f(m)) +\sum_{\alpha}x_{\alpha}^*f(x_{\alpha}m).
\end{align}
Similarly, $\hat{f}$ is given by $\hat{f}(a \otimes n) = af(n)$, and hence we have
\begin{align*}
d_{\N} \circ \hat{f} (1 \otimes m) &=d_{\N}(f(m))
\end{align*}
and
\begin{align*}
\hat{f} \circ d_{G(\M)} (1 \otimes m ) &= \hat{f}((d_{\A \otimes \M}+d'')(a\otimes m))\\
&=\hat{f} \left(1 \otimes d_{\M}(m) - \sum_{\alpha}x_{\alpha}^* \otimes x_{\alpha}m \right)\\
&=f(d_{\M}(m)) - \sum_{\alpha} x_{\alpha}^*f(x_{\alpha}m),
\end{align*}
where $d''$ denotes the Koszul differential of $G(\M) =\A \otimes^K \M$.Hence (3.4) is equivalent to
\begin{align*}
d_{\N}(f(m))=f(d_{\M}(m)) - \sum_{\alpha} x_{\alpha}^*f(x_{\alpha}m),
\end{align*}
which is the same equation as (3.5). Thus (3.3) and (3.4) are equivalent and this proves the claim.
\end{proof}
\begin{lem}\label{lem:sur}
Let $\M \in D^+(\B)$. Then there is a quasi-isomorphism of dg-$\B$-modules $\M \cong F \circ G (\M)$.
\end{lem}
\begin{proof}
Let $\M \in D^+(\B)$. First we define a homomorphism $\varphi : \M \to F \circ G (\M) \cong \K_{\A} \otimes_{\str}^K \M$ as the unit map of adjoint functors in Lemma \ref{lem:adj}. To prove that $\varphi$ is a quasi-isomorphism, consider a homomorphism of dg-$\str$-modules
\begin{align*}
\psi :\K_{\A} \otimes^K \M \to \M,
\end{align*}
which is defined by the projection onto $0$-th degree. Then, $\psi \circ \varphi = \operatorname{id}_{\M}$, so it is enough to show that $\psi$ is a quasi-isomorphism. Now, $ \K_{\A} \otimes^K \M$ is the total complex of a double complex
\begin{align*}
0 \to \K_{\A,0} \otimes_{\str} \M \to \K_{\A,-1} \otimes_{\str} \M \to \K_{\A,-2} \otimes_{\str} \M \to\cdots.
\end{align*}
where the horizontal differentials are given by the Koszul differentials, and note that each $\K_{\A,p} \otimes_{\str} \M$ forms a bounded-below complex with the usual differential of tensor products. From this, we obtain a covergent spectral sequence (see e.g. \cite[\S 5.6]{Wei})
\begin{align}
E^{p,q}_2 = H^q(\K_{\A,p} \otimes_{\str} \M) \Rightarrow H^{p+q} (\K_{\A} \otimes^K \M).
\end{align}
Let us calculate $H^q(\K_{\A,p} \otimes_{\str} \M)$. Note that $\K_{\A} \to \str$ gives a quasi-isomorphism of graded dg-algebras, and hence $\K_{\A,p}$ is acyclic for $p \neq 0$ and $\K_{\A,0} \to \str$ is a quasi-isomorphism. Since $\K_{\A,p}$ is a complex consisting of locally free sheaves, for $p \neq 0$ we have
\begin{align*}
H^q(\K_{\A,p} \otimes_{\str} \M) = 0,
\end{align*}
and
\begin{align*}
H^q(\K_{\A,0} \otimes_{\str} \M) &= H^q(\K_{\A,0} \otimes_{\str}^{\mathbb{L}} \M)\cong H^q(\str \otimes_{\str}^{\mathbb{L}} \M)\cong H^q(\M).
\end{align*}
Thus we get $H^p(\K_{\A} \otimes^K \M) \cong H^p(\M)$ for all $p \in \mathbb{Z}$ by the spectral sequence (3.6).
Moreover, this isomorphism is induced by a morphism of double complexes
\begin{align*}
\xymatrix{
	0 \ar@{>}[r] \ar@{>}[d]&\K_{\A,0}\otimes \M \ar@{>}[r] \ar@{>}[d] &\K_{\A,-1}\otimes \M \ar@{>}[r] \ar@{>}[d]&\K_{\A,-2} \otimes \M \ar@{>}[r] \ar@{>}[d]& \cdots \\
	0 \ar@{>}[r]&\ \ \M \ \ \ar@{>}[r] &\ \ 0\ \  \ar@{>}[r] &\ \ 0 \ \  \ar@{>}[r] & \cdots
}
\end{align*}
Moreover, taking total complexes of these double complexes we obtain $\psi$. Hence $\psi$ is a quasi-isomorphism, and so is $\varphi$.
\end{proof}

\begin{lem}\label{lem:bound}
The functor $F:D(\A) \to D(\B)$ preserves bounded-below complexes. In other words, $F$ induces a functor
\begin{align*}
F:D^+(\A) \to D^+ (\B).
\end{align*}
\end{lem}
\begin{proof}
Let $\M \in D^+(\A)$. Then, as a complex of $\str$-modules, we have
\begin{align*}
F(\M) \cong \K_{\A} \otimes_{\A} \M \cong \B^{\vee} \otimes^K \M,
\end{align*}
and the complex $\B^{\vee} \otimes^K \M$ is a total complex of
\begin{align*}
\cdots \to \B_n^{\vee} \otimes_{\str} \M \to \B_{n-1}^{\vee} \otimes_{\str} \M \to \B_1^{\vee} \otimes_{\str} \M \to \B_0^{\vee} \otimes_{\str} \M \to 0,
\end{align*}
where the horizontal arrows denote the Koszul differentials. Then we have a convergent spectral sequence
\begin{align*}
E_1^{p,q}=H^q(\B_{-p}^{\vee} \otimes_{\str} \M) \Rightarrow H^{p+q}(F(\M)).
\end{align*}
Since $\M \in D^+(\A)$, there exists $N \in \mathbb{Z}$ satisfying $H^q(\M)=0$ for $q < N$. On the other hand, we have $H^i(\B_p^{\vee})=0$ for $i<-p$, and hence $H^q(\B_{-p}^{\vee} \otimes_{\str} \M) = 0 $ for $q < N-p$. From the spectral sequence above, we obtain $H^p(F(\M))=0$ for $p<N$. Thus $F(\M) \in D^+(\B)$.
\end{proof}
\begin{rem}
This lemma does not follows from lemma \ref{lem:F} since $\A$ is not connective.
\end{rem}
Now, we are ready to prove Theorem \ref{thm:main}.
\begin{proof}
We will only prove the ungraded case. The graded case can be proven similarly. First, consider a functor
\begin{align*}
F : \D(\A) \to \D(\B),\ \  \M \mapsto \K_{\A} \otimes_{\A} \M.
\end{align*}
Now, there is an equivalence of $\infty$-categories
\begin{align*}
\D(\A) \cong \lim_{U \in \operatorname{AffOp(X)}} \D(\A|_U),\\
\D(\B) \cong \lim_{U \in \operatorname{AffOp(X)}}\D(\B|_U),
\end{align*}
where $\operatorname{AffOp(X)}$ denotes the category of affine open sets of $X$, and $F$ is induced by
\begin{align*}
F_U:\D(\A|_U) \to \D(\B|_U), \M \mapsto \K_{\A}|_U \otimes_{\A|_U} \M.
\end{align*}
Hence it is enough to show that $F_U$ induces equivalences 
\begin{align*}
\D(\A|_U) &\overset{\sim} \to \operatorname{IndCoh}_U\B|_U, \\
\D^+(\A|_U) &\overset{\sim} \to \D^+(\B|_U),
\end{align*}
for any affine open sets $U \subset X$. Thus we may assume that $X$ is affine.

Consider the homomorphism induced by $F$
\begin{align*}
\A \cong \operatorname{End}_{\A}(\A) \to \operatorname{REnd}_{\B}(\K_{\A}).
\end{align*}
Now we have
\begin{align*}
\operatorname{REnd}_{\B}(\K_{\A}) &\cong \operatorname{REnd}_{\B}(\str)\\
&\cong \operatorname{Hom}_{\B}(\K_{\B},\str)\\
&\cong \A.
\end{align*}
Hence the above homomorphism is a quasi-isomorphism, and thus we have an equivalence
\begin{align*}
\operatorname{Perf} \A \overset{\sim}{\to} \operatorname{Coh}_X \B.
\end{align*}
Taking ind-completion, we obtain
\begin{align*}
\D(\A) \overset{\sim}{\to} \operatorname{IndCoh}_X \B.
\end{align*}
From this, we obtain a fully faithful functor 
\begin{align*}
\D^+(\A) \hookrightarrow \operatorname{IndCoh}^+_X \B \hookrightarrow \operatorname{IndCoh}^+ \B \cong \D^+(\B)
\end{align*}
by Lemma \ref{lem:bound}. Here, the symbols $^+$ denote the bounded-below subcategories. Hence it is enough to show that this functor is essentially surjective, but this immediately follows from Lemma \ref{lem:sur}.
\end{proof}

\section{Applications to Derived Loop Spaces}
In this section, we give an application of Theorem \ref{thm:main} to the derived category of derived loop spaces. Hereafter, we will always assume that the characteristic of $k$ is 0.
\subsection{Koszul Duality for Derived Loop Spaces}
The derived loop space $\mathcal{L}X$ of a derived scheme $X$ is a derived algebraic analogue of free loop spaces defined to be $\operatorname{Map}(S^1,X)$ where $S^1$ is a constant functor valued at $\Delta^1 \coprod_{\{0\} \coprod \{1\}} \Delta^1 \cong * \coprod_{* \coprod *}*$. By definition, this can be written as $\mathcal{L}X = X \times_{X \times X} X$. If $X=\operatorname{Spec} A$ is a smooth affine scheme, then $\mathcal{L}X = \operatorname{Spec} A \otimes_{A \otimes_k A}^{\mathbb{L}} A$, which is isomorphic to $\operatorname{Spec} \Sym (\Omega_{A/k}[1])$ by Hochschild-Kostant-Rosenberg isomophism \cite{HKR}. Ben-Zvi and Nadler \cite{BZ} generalized this isomorphism to general schemes in terms of the cotangent complex. More generally, for any $n \in \mathbb{Z}_{> 0}$ , we can also define the $n$-fold derived loop space $\mathcal{L}^{(n)}X$ by $\mathcal{L}^{(n)}X := \operatorname{Map}(S^n,X)$. Here, note that we have $\mathcal{L}^{(1)}X = \mathcal{L}X$. Since we have an isomorphism $H^{\bullet}(S^n,k) \cong k \oplus k[-n]$, Ben-Zvi and Nadler\cite{BZ} essentially proves the following descriptions of $n$-fold derived loop spaces.
\begin{prop} (cf. \cite[Proposition 4.4]{BZ}) Let $X$ be a separated, quasi-compact scheme over $k$ and let $n \in \mathbb{Z}_{>0}$. There is an isomorphism
\begin{align*}
\mathcal{L}^{(n)}X \cong \operatorname{Spec}_{\str}\Sym (\mathbb{L}_{X/k}[n]).
\end{align*}
\end{prop}

Now, we will apply Theorem \ref{thm:main} to the derived category of loop spaces $\D(\mathcal{L}^{(n)}X) \cong \D(\Sym (\mathbb{L}_{X/k}[n]))$. To do this, first we need to check that $\mathbb{L}_{X/k}[n]$ satisfies the assumptions of Theorem \ref{thm:main}, at least affine locally. To see this, we will use Lurie's criterion.
\begin{dfn}
Let $A$ be a connective commutative dg-algebra over $k$, and let $\operatorname{CAlg}_k^{\geq n}$ denote the $\infty$-category of connective commutative dg-algebras $B$ over $k$ satisfying $H^i(B)=0$ for $i < n$. Similarly, we define $\D(A)^{\geq n}$ to be the $\infty$-category of dg-$A$-modules $M$ satisfying $H^i(M)=0$ for $i < n$.
\begin{enumerate}
\item We say that $A$ is almost of finite presentation over $k$ if the truncation $\tau_{\geq n} A$ is a compact object of $\operatorname{CAlg}_k^{\geq n}$ (that is, $\operatorname{Map}(\tau_{\geq n}A,-)$ commutes with filtered colimits) for all $n \leq 0$, where the truncation $\tau_{\geq n} A$ is defined by
\begin{align*}
(\tau_{\geq n} A)^i=
\begin{cases}
A^i &(i>n),\\
\operatorname{Coker} d_A^{n-1} \ \ &(i=n),\\
0 &(i<n).
\end{cases}
\end{align*}
\item Let $M$ be a dg-$A$-module. We say that $M$ is almost perfect if $M$ satisfies the following conditions:
\begin{itemize}
\item $M$ is bounded above. Namely, $H^i(M)=0$ for sufficiently large $i$.
\item For any $n \in \mathbb{Z}_{<0}$, the truncation $\tau_{\geq n}M$ is a compact object of $\D(A)^{\geq n}$.
\end{itemize}
\end{enumerate}
\end{dfn}

\begin{thm} \label{thm:Lur}( \cite[Theorem 7.4.3.18.]{Lur})
Let $A$ be a connective commutative dg-algebra over $k$. If $A$ is almost of finite presentation over $k$, then $\mathbb{L}_{A/k}$ is almost perfect.
\end{thm}

\begin{cor} \label{cor:cotan}
Let $A$ be a $k$-algebra of finite type. Then the cotangent complex $\mathbb{L}_{A/k}$ is quasi-isomorphic to a connective complex consisting of free $A$-modules of finite rank.
\end{cor}
\begin{proof}
Since $A$ is (ordinary) $k$-algebra of finite type, $A$ is almost of finite presentation over $k$ in the sense defined above (cf. \cite[Proposition 7.2.4.31]{Lur}), and hence $\mathbb{L}_{A/k}$ is almost perfect by Theorem \ref{thm:Lur}. Then each $H^i(\mathbb{L}_{A/k})$ is finitely generated $A$-module (cf. \cite[Proposition 7.2.4.17.]{Lur}). From this, we obtain the claim by taking a free resolution of $\mathbb{L}_{A/k}$.
\end{proof}

Before applying Theorem \ref{thm:main} to the derived category of derived loop spaces, we prepare some notations.

\begin{dfn}
Let $\A$ be a sheaf of graded dg-algebras on $X$. The category $D_{gr}^{\searrow}(\A)$ denotes the full subcategory of $D_{gr}(\A)$ consisting of $\M \in D_{gr}(\A)$ satisfying $H^i(\M)_j = 0$ for $i-2j \ll 0$.
\end{dfn}

Now, we provide applications of Theorem \ref{thm:main} to derived loop spaces.

\begin{thm}\label{thm:loop}
Let $X$ be a separated scheme of finite type over $k$ and let $n \in \mathbb{Z}_{>0}$. Then we have equivalences of $\infty$-categories
\begin{align*}
\operatorname{IndCoh}_X&\mathcal{L}^{(n)}X \cong \D(\Sym (\mathbb{T}_{X/k}[-n-1])),\\
\operatorname{IndCoh}_{gr,X}&\mathcal{L}^{(n)}X \cong \D_{gr}(\Sym (\mathbb{T}_{X/k}[-n-1])).
\end{align*}
If we restrict to bounded-below subcategories, we obtain equivalences of triangulated categories
\begin{align*}
D^+(\mathcal{L}^{(n)}X) \cong D^+(\Sym (\mathbb{T}_{X/k}[-n-1])), \\
D_{gr}^+(\mathcal{L}^{(n)}X) \cong D_{gr}^+(\Sym (\mathbb{T}_{X/k}[-n-1])),
\end{align*}
where $\Sym (\mathbb{L}_{X/k}[n])$ is regarded as a sheaf of graded dg-algebras by considering the sections of  $\mathbb{L}_{X/k}[n]$ to have internal degree $1$.
\end{thm}
\begin{proof}
We only prove the ungraded case. The graded case can be proven similarly. For simplicity, we set $\A:=\Sym (\mathbb{T}_{X/k}[-n-1])$ and $\B:=\Sym (\mathbb{L}_{X/k}[n])$. Note that we have equivalences of $\infty$-categories
\begin{align*}
\operatorname{IndCoh}(\B) &\cong \lim_{U \in \operatorname{AffOp}(X)} \operatorname{IndCoh}(\B|_U), \\
\D(\A) &\cong \lim_{U \in \operatorname{AffOp}(X)} \D(\A|_U).
\end{align*}
First, by \cite[Theorem 4.8.5.11]{Lur}, we have an adjunction of $\infty$-categories
\begin{align*}
\theta : \operatorname{Alg}_k 
\mathrel{\substack{
  \longrightarrow\\[-0.7ex]
  \perp\\[-0.7ex]
  \longleftarrow
}}
(\operatorname{Pr}^L_k)_{\D(k)/} : E
\end{align*}
where $\operatorname{Alg}_k$ denotes the $\infty$-category of dg-algebras over $k$, and $\operatorname{Pr}^L_k$ denotes the $\infty$-category of $k$-linear presentable $\infty$-categories with colimit-preserving functors, and $\theta$ and $E$ are given by  $\theta(A) :=(\D(A),A)$ and $E(\mathcal{C},C) := \operatorname{REnd}_{\mathcal{C}}(C)$, respectively.
Consider a diagram in $ (\operatorname{Pr}^L_k)_{\D(k)/}$ indexed by affine open sets of $X$
\begin{align*}
\operatorname{AffOp}(X) \ni U \longmapsto (\operatorname{IndCoh}\B|_U,\mathcal{O}_X|_U) \in (\operatorname{Pr}^L_k)_{\D(k)/}
\end{align*}
with pullback functors. Then, from the counit map of the above adjunction, we obtain a functor of $\infty$-categories
\begin{align}
\lim_{U \in \operatorname{AffOp}(X)} \D(\operatorname{REnd}_{\B|_U}(\mathcal{O}_X|_U)) \longrightarrow \lim_{U \in \operatorname{AffOp}(X)} \operatorname{IndCoh}\B|_U.
\end{align}
We will investigate the functor induced by the counit map
\begin{align*}
(\D(\operatorname{REnd}_{\B}(\mathcal{O}_X|_U)),\operatorname{REnd}_{\B}(\mathcal{O}_X|_U)) \to (\operatorname{IndCoh}(\B|_U),\mathcal{O}_X|_U) \cong (\operatorname{IndCoh}(\B|_U),\K_{\A|_U}).
\end{align*}
Here, by Corollary \ref{cor:cotan}, $\B|_U$ is quasi-isomorphic to a symmetric algebra of a connective complex consisting of locally free sheaves of finite rank, and hence we can define $\K_{\A|_U}$ and $\K_{\B|_U}$. Thus, by the calculations in the proof of Theorem \ref{thm:main}, we have 
\begin{align*}
(\D(\operatorname{REnd}_{\B}(\mathcal{O}_X|_U)),\operatorname{REnd}_{\B|_U}(\mathcal{O}_X|_U))\cong (D(\A|_U),\A|_U),
\end{align*}
and the functor $D(\A|_U) \to \operatorname{IndCoh}\B|_U$ induced by the counit map is given by $\M \mapsto \K_{\A|_U} \otimes_{\A|_U} \M$, since it is a colimit-preserving functor which sends $\A$ to $\K_{\A|_U}$. Hence this induces equivalences
\begin{align*}
\D(\A|_U) &\overset{\sim}\to \operatorname{IndCoh}_U\B|_U,\\
\D^+(\A|_U) &\overset{\sim}\to \D^+(\B|_U).
\end{align*}
by Theorem \ref{thm:main}. Hence the functor (4.1) gives an equivalence
\begin{align*}
\lim_{U \in \operatorname{AffOp}(X)} \D(\A|_U) \overset{\sim}\longrightarrow \lim_{U \in \operatorname{AffOp}(X)} \operatorname{IndCoh}_U\B|_U.
\end{align*}
Moreover, for affine open sets $V \subset U \subset X$, the functor $\D(\A|_U) \to \D(\A|_V)$ is given by the pullback functor. Hence we obtain an equivalence
\begin{align*}
\D(\A) \cong \lim_{U \in \operatorname{AffOp}(X)} \D(\A|_U) \overset{\sim}\longrightarrow \lim_{U \in \operatorname{AffOp}(X)} \operatorname{IndCoh}_U\B|_U \cong \operatorname{IndCoh}_X\B.
\end{align*}
Restricting to the bounded-below subcategories, we also have $\D^+(\A) \cong \D^+(\B)$.
\end{proof}
As a corollary, we obtain the following equivalences between $\mathcal{L}X$ and $\Sym \mathbb{T}_{X/k}$.
\begin{cor}\label{cor:loop}
Let $X$ be a separated scheme of finite type over $k$. Then we have an equivalence of $\infty$-categories
\begin{align*}
\operatorname{IndCoh}_{gr,X}&\mathcal{L}X \cong  \D_{gr}(\Sym \mathbb{T}_{X/k}).
\end{align*}
If we restrict to bounded-below subcategories, we obtain an equivalence of triangulated categories
\begin{align*}
D^+_{gr}(\mathcal{L}X) \cong D_{gr}^{\searrow}(\Sym \mathbb{T}_{X/k}).
\end{align*}
\end{cor}
\begin{proof}
First, by Theorem \ref{thm:loop}, we obtain
\begin{align*}
\operatorname{IndCoh}_{gr,X}\mathcal{L}X&\cong\D_{gr}(\Sym (\mathbb{T}_{X/k}[-2])),\\
D_{gr}^+(\mathcal{L}X)&\cong D_{gr}^+(\Sym (\mathbb{T}_{X/k}[-2])).
\end{align*}
On the other hand, we have an equivalence
\begin{align*}
C_{gr}(\Sym \mathbb{T}_{X/k}) \cong C_{gr}(\Sym (\mathbb{T}_{X/k}[-2])),
\end{align*}
which is given by $\mu:C_{gr}(\Sym (\mathbb{T}_{X/k}[-2])) \to C_{gr}(\Sym \mathbb{T}_{X/k})$, $\mu(\M)^i_j=\M^{i-2j}_j$. Hence we have 
\begin{align*}
\D_{gr}(\Sym (\mathbb{T}_{X/k}[-2])) \cong \D_{gr}(\Sym \mathbb{T}_{X/k}),
\end{align*}
and this proves the claim.
\end{proof}

\subsection{Relation to Derived Fromal Stacks}
In this section, we give a geometrical interpretation of the equivalence we obtained in the previous section, by relating the dg-algebra $\Sym (\mathbb{T}_{X/k}[-2])$ and some derived formal stacks. First, we briefly recall the theory of derived formal stacks, developed by Lurie \cite{DAGX} and Hennion \cite{Henn}.

First, we prepare some notations.
\begin{dfn}
Let $A$ be a Noetherian ring over $k$.
\begin{enumerate}
\item Let $\mathcal{S}$ denotes the $\infty$-category of Kan complexes.
\item Let $\operatorname{CAlg}_A$ denote the $\infty$-category of connective commutative dg-algebras over $A$. Let $\operatorname{CAlg}^{sm}_A$ denote the full subcategory of $(\operatorname{CAlg}_A)_{/A}$ spanned by the trivial square-zero extensions $A \oplus M$, where $M$ is a finitely generated free dg-$A$-module concentrated in non-positive degree.
\item Let $\operatorname{dgLie}_A$ denote the $\infty$-category of dg-Lie-algebras over $A$. Let $\operatorname{dgLie}_A^*$ denote the subcategory spanned by free dg-Lie algebras generated by a free dg-$A$-module which is generated by finite number of positive elements. Namely, objects of $\operatorname{dgLie}_A^*$ are free dg-Lie algebras generated by a dg-$A$-module $\bigoplus_{i=1}^r A^{\oplus n_i}[-i]$ for some $r \in \mathbb{Z}_{>0}$ and $n_i \in \mathbb{Z}_{\geq 0}$ for $i=1,2,\cdots,r$.
\end{enumerate}
\end{dfn}
\begin{dfn}
Let $A$ be a Noetherian ring over $k$. A derived formal stack over $A$ is a functor
\begin{align*}
Y:\operatorname{CAlg}^{sm}_A \to \mathcal{S}
\end{align*}
satisfying the following conditions:
\begin{enumerate}
\item $Y$ preserves finite products.
\item For any $B \in \operatorname{CAlg}^{sm}_A$, there is a homotopy equivalence $Y(A \times_B A) \simeq * \times_{Y(B)} *$.
\end{enumerate}
Let $\operatorname{dSt}_A^f$ denote the $\infty$-category of derived formal stacks over $A$.
\end{dfn}
\begin{rem}
This definition is different from the one of formal moduli problems in \cite{DAGX}. However, it is proved that these two definitions are equivalent \cite[Proposition 1.5.10]{Henn}.
\end{rem}
We will extend this definition to arbitrary derived schemes. Let $A$ and $B$ be connective dg-algebras over $k$, and let $A \to B$ a homomorphism of dg-algebras. Then, we have a functor
\begin{align*}
(\operatorname{CAlg}^{sm}_A)^{op}\to (\operatorname{CAlg}^{sm}_B)^{op} \to \operatorname{dSt}_B^f,
\end{align*}
composing Yoneda embedding and scalar extensions. This functor admits a left Kan extension $\operatorname{dSt}^f_A \to \operatorname{dSt}^f_B$. Using this, we define derived formal stack over derived schemes.
\begin{dfn}
Let $X$ be a derived scheme over $k$. The $\infty$-category of derived formal stacks over $X$ is defined to be the limit
\begin{align*}
\operatorname{dSt}^f_X = \lim_{\operatorname{Spec}A \to X} \operatorname{dSt}^f_A.
\end{align*}
\end{dfn}
Derived formal stacks and dg-Lie algebras are related as follows. First we define the Chevalley-Eilenberg algebras of dg-Lie algebras.
\begin{dfn}
Let $A$ be a Noetherian ring over $k$, and let $M:=\bigoplus_{i=1}^r A^{\oplus n_i}[-i]$ be a free dg-$A$-module generated by positive elements. Consider a free dg-Lie algebra $L$ generated by $M$. The Chevalley-Eilenberg algebra of $L$ is given by $C_A(L) := A \oplus M^{\vee}[-1]$, the trivial square-zero extension of $A$ by $M^{\vee}[-1]$.
\end{dfn}
Using this, we obtain a functor $\operatorname{dgLie}^*_A \to \operatorname{dSt}_A^f$ which sends $L$ to $\operatorname{Map}(C_A(L),-)$. Since $\operatorname{dgLie}_A$ is the sifted completion of $\operatorname{dgLie}^*_A$ (cf. \cite[Proposition 1.2.2.]{Henn}), we have a functor 
\begin{align*}
\mathcal{F}_A:\operatorname{dgLie}_A \to \operatorname{dSt}_A^f
\end{align*}
as the left Kan extension. Moreover, the functor $\mathcal{F}_A$ also induces a functor
\begin{align*}
\mathcal{F}_X : \operatorname{dgLie}_X \to \operatorname{dSt}_X^f. 
\end{align*}
Through this functor, we can relate the derived category of derived formal stacks and the derived category of representations of corresponding dg-Lie algebras. To see this, we will first define the derived category of formal stacks.
\begin{dfn}
Let $A$ be a Noetherian ring, and let $Y$ be a derived formal stack over $A$. Then we define the $\infty$-categories
\begin{align*}
\D(Y) &= \lim_{\substack{\operatorname{Spec}B \to Y \\ B \in \operatorname{CAlg}_A^{sm}}} \D(B),\\
\operatorname{IndCoh}Y &= \lim_{\substack{\operatorname{Spec}B \to Y \\ B \in \operatorname{CAlg}_A^{sm}}} \operatorname{IndCoh}(B).
\end{align*}
\end{dfn}
\begin{rem}
Let $\mathfrak{g}$ be a dg-Lie algebra over $A$, and let $Y=\mathcal{F}_A(\mathfrak{g})$ be the corresponding derived formal stack. We can write $\mathfrak{g}$ as the following colimit
\begin{align*}
\mathfrak{g} = \operatorname*{colim}_{L \in K} L
\end{align*}
where $K$ is a sifted diagram in $\operatorname{dgLie_A^*}$. Then we have
\begin{align*}
\operatorname{IndCoh}Y = \lim_{L \in K} \operatorname{IndCoh}C_A(L).
\end{align*}
Note that for $L,L' \in K$ and $L \to L'$, the !-pullback functor induces a functor $\D^+(C_A(L)) \to \D^+(C_A(L'))$. Hence we also have
\begin{align*}
\D^+(Y) = \lim_{L \in K} \D^+(C_A(L)).
\end{align*}
If moreover $\mathfrak{g}$ is a graded dg-Lie algebra, we also have graded version of this
\begin{align*}
\operatorname{IndCoh}_{gr}Y &= \lim_{L \in K} \operatorname{IndCoh}_{gr}C_A(L),\\
\D^+_{gr}(Y)& = \lim_{L \in K} \D_{gr}^+(C_A(L)).
\end{align*}
since each $L$ can be taken as a graded dg-Lie algebra.
\end{rem}
\begin{dfn}
Let $X$ be a scheme of finite type over $k$, and let $Y$ be a derived formal stack over $X$. Then we define the $\infty$-categories
\begin{align*}
\D(Y) &= \lim_{U \in \operatorname{AffOp}(X)} \D(Y_U),\\
\operatorname{IndCoh}Y &= \lim_{U \in \operatorname{AffOp}(X)} \operatorname{IndCoh}(Y_U),
\end{align*}
where $Y_U$ denotes the pullback of $Y$ on $U$. 
\end{dfn}
Now, we will establish the relation between dg-Lie algebras and corresponding derived formal stacks. Let $A$ be a Noetherian ring over $k$, and let $M:=\bigoplus_{i=1}^r A^{\oplus n_i}[-i]$ be a free dg-$A$-module generated by positive elements. Consider a free dg-Lie algebra $L$ generated by $M$. Then, its universal enveloping algebra over $A$ is given by $U_A(L)=T_A(M)$, the tensor algebra of $M$ over $A$. In this situation, we have the following equivalence of categories:
\begin{prop}\label{prop:lie}
Let $X=\operatorname{Spec}A$ be an affine scheme of finite type over $k$, and let $L$ be a free dg-Lie algebra generated by $M:=\bigoplus_{i=1}^r A^{\oplus n_i}[-i]$. Then, there is an equivalence of $\infty$-categories
\begin{align*}
\D(U_A(L)) \cong \operatorname{IndCoh}_XC_A(L).
\end{align*}
Here, $\D(U_A(L))$ denotes the $\infty$-category of left dg-$U_A(L)$-modules and $\operatorname{IndCoh}_XC_A(L)$ denotes the ind-completion of $\operatorname{Coh}_XC_A(L)$, which is the thick subcategory of $\D(C_A(L))$ generated by $A$ (cf. Definition \ref{def:cohx}). Moreover, this induces an equivalence
\begin{align*}
\D^+(U_A(L)) \cong \D^+(C_A(L)).
\end{align*}
If moreover $L$ is a graded dg-Lie algebra, we also have equivalences
\begin{align*}
\D_{gr}(U_A(L)) &\cong \operatorname{IndCoh}_{gr,X}C_A(L),\\
\D_{gr}^+(U_A(L)) &\cong \D_{gr}^+(C_A(L)).
\end{align*}
\end{prop}
To prove this, we first investigate resolutions of $A$ over $U_A(L)$ and $C_A(L)$. Consider an exact sequence of right dg-$U_A(L)$-modules
\begin{align*}
0 \to M \otimes_A U_A(L) \to U_A(L) \to A \to 0,
\end{align*}
where the first map $M \otimes_A U_A(L) \to U_A(L)$ is given by the left multiplication of $U_A(L) \cong T_A(M)$, and the second map is a natural surjection. Then, we define a right dg-$U_A(L)$-module $T$ as a total complex
\begin{align}
T:= \operatorname{Tot}(0 \to M \otimes_A U_A(L) \to U_A(L)\to 0),
\end{align}
where $U_A(L)$ is in the degree $0$ and $M \otimes_A U_A(L)$ is in the degree $-1$. More precisely, as a graded module,
\begin{align*}
T = A \otimes_A U_A(L) \oplus M[1]\otimes_A U_A(L) \cong (A \oplus M[1]) \otimes_A U_A(L),
\end{align*}
and for $a \in A$, $m \in M[1]$, and $t \in U_A(L)$, its differential is given by
\begin{align*}
d((a + m) \otimes t) = mt.
\end{align*}
Note that $T = (C_A(L))^{\vee} \otimes_A U_A(L)$ as a graded module and hence $T$ also has a structure of a dg-$C_A(L)$-module.
\begin{lem}\label{lem:reso}
The following assertions hold:
\begin{enumerate}
\item There is a quasi-isomorphism of dg-$U_A(L)$-modules $T \to A$.
\item The right dg-$U_A(L)$-module $T$ is K-projective.
\item There is a quasi-isomorphism of dg-$C_A(L)$-modules $A \to T^{\vee}$.
\item The dg-$C_A(L)$-module $T^{\vee}$ is K-projective.
\end{enumerate}
\end{lem}
\begin{proof}
(1) immediately follows from the construction. We will prove (2). Since $T \supset U_A(L)$ is a dg-submodule and $T/U_A(L) \cong M \otimes_A U_A(L)[1]$, it follows that $T$ is K-projective right dg-$U_A(L)$-module. For (3), since $T$ is K-projective as a dg-$A$-module, we obtain a quasi-isomorphism $A \cong A^{\vee} \to T^{\vee}.$

To prove (4), we first calculate the differential of $T^{\vee}$. As a graded $A$-module, we have
\begin{align*}
T \cong (C_A(L))^{\vee} \otimes_A U_A(L).
\end{align*}
Let $\{x_i\}$ be a basis of $M$ as a free dg-$A$-module. Then, for any $(a+m) \otimes t \in T \cong (C_A(L))^{\vee} \otimes U_A(L)$, where $a \in A$, $m \in M[1]$, and $t \in U_A(L)$, we have
\begin{align*}
d((a+m) \otimes t ) = 1 \otimes mt.
\end{align*}
Then, the differential of $T^{\vee} = C_A(L) \otimes_A (U_A(L))^{\vee}$ is given by
\begin{align*}
d(s \otimes \varphi) = (-1)^{|\varphi|} \sum_i sx_i^* \otimes x_i \varphi\ \ \ \cdots(\star)
\end{align*}
for a homogeneous element $s \otimes \varphi \in C_A(L) \otimes_A (U_A(L))^{\vee}$. Indeed, for any $(a+m) \otimes t \in T = (C_A(L))^{\vee} \otimes_A U_A(L)$ ,where $a \in A$, $m \in M$ and $t \in U_A(L)$, we have
\begin{align*}
&d_{T^{\vee}}(s \otimes \varphi)((a+m) \otimes t)\\
&=-(-1)^{|s|+|\varphi|} (s \otimes \varphi)(d_T((a+m) \otimes t))\\ 
&= -(-1)^{|s|+|\varphi|} (s \otimes \varphi) (1 \otimes mt)\\
&= -(-1)^{|s|+|\varphi|} s \otimes \varphi(mt)\\
&= -(-1)^{|s|+|\varphi|} s \otimes \varphi \left( \sum_i x_i^*(m)x_it \right)\\
&= -(-1)^{|s|+|\varphi|} \sum_i sx_i^*(m) \otimes (x_i\varphi)(t)\\
&= -(-1)^{|s|+|\varphi|} \left(\sum_i sx_i^* \otimes x_i\varphi \right)(m \otimes t).
\end{align*}
Here, if $s \in M^{\vee}$, then $sx_i^*=0$ and hence we obtain ($\star$).

Hence $T^{\vee} \cong C_A(L) \otimes_A (U_A(L))^{\vee}$ can be expressed as
\begin{align*}
0 \leftarrow C_A(L) \otimes_A (T^0_A(M))^{\vee} \leftarrow C_A(L) \otimes_A (T^1_A(M))^{\vee} \leftarrow \cdots \leftarrow C_A(L) \otimes_A (T^i_A(M))^{\vee}  \leftarrow \cdots
\end{align*}
where $T_A^i(M)$ denotes the $i$-th graded piece of $T_A(M)$. Thus, if we define $F^n:=\bigoplus_{i=0}^n C_A(L) \otimes_A (T^i_A(M))^{\vee}$, the filtration $\{F^n\}$ gives a cellularly split filtration of $T^{\vee}$. Hence $T^{\vee}$ is a K-projective dg-$C_A(L)$-module by Lemma \ref{lem.spa}.
\end{proof}
Consider a functor $F$ given by
\begin{align*}
F:\D(U_A(L)) \to \D(C_A(L)), N \mapsto T \otimes_{U_A(L)}N.
\end{align*}
By Lemma \ref{lem:reso}, $F$ is well-defined. We will show that this functor induces equivalences $\D(U_A(L)) \cong \operatorname{IndCoh}C_A(L)$ and $\D^+(U_A(L)) \cong \D^+(C_A(L))$. To see this, we prepare some lemmas. First, we will define Koszul differentials in this situation. For a right dg-$U_A(L)$-module $N$ and a dg-$C_A(L)$-module $P$, we define a dg-$A$-module $N \otimes_A^K P$ as follows. First, as a graded module, $N \otimes_A^KP:= N \otimes_AP$. Its differential is given by
\begin{align*}
d_{N \otimes_A^KP} := d_{N \otimes_A P} + d',
\end{align*}
where $d'$ is the Koszul differential defined by
\begin{align*}
d'(n \otimes p) := -(-1)^{|n|}\sum_i nx_i \otimes x_i^*p,
\end{align*}
for a homogeneous element $n \otimes p \in N \otimes_A P$.
\begin{lem}\label{lem:sur2}
Let $N$ be a dg-$C_A(L)$-module which is bounded below. Then we have a quasi-isomorphism of dg-$C_A(L)$-modules
\begin{align*}
N \overset{\sim}\to T \otimes_A^K N,
\end{align*}
where $C_A(L)$-action on $T \otimes_A^K N$ is given by the one on $T$.
\end{lem}
\begin{proof}
First, we have an isomorphism
\begin{align*}
\operatorname{Hom}_{C_A(L)}(N, (C_A(L))^{\vee} \otimes U_A(L) \otimes_AN) &\cong \operatorname{Hom}_{U_A(L)}(U_A(L)\otimes_AN, U_A(L) \otimes_AN).
\end{align*}
Take a morphism $\varphi: N \to (C_A(L))^{\vee} \otimes U_A(L) \otimes_AN$ corresponding to $\operatorname{id}_{U_A(L) \otimes_A N}$. We will show that this gives a quasi-isomorphism $N \overset{\sim}\to T \otimes_A^K N$. Let us check that $\varphi$ commutes with differentials. To see this, it is enought to show that $(d_T+d') \circ \varphi=0$, where $d'$ is the Koszul differential of $T \otimes_A^KN$. First, $\varphi$ is given by
\begin{align*}
\varphi(n) = \sum_i (-1)^{|x_i||n|} x_i \otimes 1 \otimes x_i^*n
\end{align*}
for a homogeneous element $n \in N$. Then,
\begin{align*}
d_T (\varphi(n))= \sum_i (-1)^{|x_i||n|} 1 \otimes x_i \otimes x_i^*n.
\end{align*}
On the other hand, since $C_A(L) = A \oplus M^{\vee}[-1]$ is a square-zero extension, we have
\begin{align*}
d' (\varphi(n))&= -\sum_{i,j} (-1)^{|x_i||n|} x_i \otimes x_j \otimes x_j^*x_i^*n \\
&= -\sum_i (-1)^{|x_i||n|} 1 \otimes x_i \otimes x_i^*n 
\end{align*}
Hence we have $(d_T+d') \circ \varphi=0$, and thus $\varphi$ gives a homomorphism of dg-$C_A(L)$-modules $N \to T \otimes_A^K N$.

We will show that $\varphi$ is a quasi-isomorphism. Consider a homomorphism of dg-$A$-modules
\begin{align*}
\psi : T \otimes_A^K N \to N, f \otimes t \otimes n \mapsto f(1)tn,
\end{align*}
where $f \in (C_A(L))^{\vee}$, $t \in U_A(L)$, and $n \in N$. Then $\psi \circ \varphi = \operatorname{id}_N$ and hence it suffices to show that $\psi$ is a quasi-isomorphism, which can be proven by the similar argument as in the proof of Lemma \ref{lem:sur}.
\end{proof}

\begin{lem}\label{lem:bound2}
The functor $F:\D(U_A(L)) \to \D(C_A(L))$ preserves bounded-below complexes. Namely, for any $N \in \D^+(U_A(L))$, we have $F(N) \in \D^+(C_A(L))$.
\end{lem}
\begin{proof}
Let $N \in \D^+(U_A(L))$. Then, as a graded module, $F(N) \cong (A \oplus M[-1]) \otimes_A N$, and its differential is given by
\begin{align*}
d((a+m)\otimes n) = mn,
\end{align*} 
where $a\in A$, $m \in M[1]$, and $n \in N$. Here, note that $m \in M \subset T_A(M) \cong U_A(L)$ acts on $N$. From this, $F(M)$ can be written as a total complex of the following double complex
\begin{align*}
0 \to M \otimes_A N \to N \to 0.
\end{align*}
Then, by the spectral sequence argument as in the proof of Lemma \ref{lem:bound}, we obtain the claim.
\end{proof}

Now, we are ready to prove Proposition \ref{prop:lie}.
\begin{proof}
We will prove only the ungraded case. The graded case can be proven similarly. Consider a functor
\begin{align*}
F:D(U_A(L)) \to D(C_A(L)), M \mapsto T \otimes_{U_A(L)} M.
\end{align*}
This functor induces a homomorphism
\begin{align*}
U_A(L) \cong \operatorname{End}_{U_A(L)}(U_A(L)) \to \operatorname{REnd}_{C_{A}(L)}(T),
\end{align*}
and we have
\begin{align*}
\operatorname{REnd}_{C_{A}(L)}(T) &\cong \operatorname{REnd}_{C_{A}(L)}(A) \\
&\cong \operatorname{Hom}_{C_A(L)} (T^{\vee},A)\\
&\cong \operatorname{Hom}_{C_A(L)} (C_A(L) \otimes_A (U_A(L))^{\vee},A)\\
&\cong \operatorname{Hom}_{A} ((U_A(L))^{\vee},A)\\
& \cong U_A(L).
\end{align*}
Hence the above homomorphism is a quasi-isomorphism. Thus we have an equivalence of categories
\begin{align*}
\operatorname{Perf}U_A(L) \overset{\sim}{\to} \operatorname{Coh}_X C_A(L).
\end{align*}
Taking ind-completions, we obtain
\begin{align*}
\D(U_A(L)) \overset{\sim}{\to} \operatorname{IndCoh}_X C_A(L).
\end{align*}
Restricting this functor to the bounded-below subcategories, we obtain a fully-faithful functor
\begin{align*}
\D^+(U_A(L)) \hookrightarrow \operatorname{IndCoh}_X^+ C_A(L) \hookrightarrow \operatorname{IndCoh}^+ C_A(L)  \cong \D^+(C_A(L)),
\end{align*}
since $F$ preserves bounded-below complexes by Lemma \ref{lem:bound2}.

Then, to prove that $F$ gives an equivalence, it suffices to show that $F$ is essentially surjective. Let $N \in D^+(C_A(L))$. Then we have
\begin{align*}
F(U_A(L) \otimes_A^K N) &\cong T \otimes_{U_A(L)}(U_A(L) \otimes_A^K N)\\
& \cong T \otimes^K_A N.
\end{align*}
Since we have a quasi-isomorphism of dg-$C_A(L)$-modules $N \overset{\sim}\to T \otimes_A^K N$ by Lemma \ref{lem:sur2}, we obtain $F(U_A(L) \otimes_A^K N) \cong N$. This shows that $F:\D^+(U_A(L)) \to \D^+(C_A(L))$ is essentially surjective.
\end{proof}
Now, we give a relation of derived categories between dg-Lie algebras and corresponding derived formal stacks. To formulate this, we prepare some notations.
\begin{dfn}\label{def:indcohx2}
Let $\mathfrak{g}$ be a dg-Lie algebra over $X$, and let $Y=\mathcal{F}_X(\mathfrak{g})$ be the corresponding derived formal stack. Write $\mathfrak{g}=\operatorname{colim}_{L \in K}L$ as the colimit of a diagram in $\operatorname{dgLie}_X^*$. Then we define
\begin{align*}
\operatorname{IndCoh}_XY=\lim_{L \in K}\operatorname{IndCoh}_XC_A(L),
\end{align*}
and we define $\operatorname{IndCoh}_{gr,X}Y$ similarly. (For the definition of $\operatorname{IndCoh}_XC_A(L)$, see Definition \ref{def:cohx}.)

For a general scheme $X$ of finite type over $k$, a dg-Lie algebra $\mathfrak{g}$ over $X$, and $Y=\mathcal{F}_X(\mathfrak{g})$, we define
\begin{align*}
\operatorname{IndCoh}_XY = \lim_{U \in \operatorname{AffOp}(X)} \operatorname{IndCoh}_UY_U,
\end{align*}
where $Y_U$ denotes the pullback of $Y$ to $U$. We define $\operatorname{IndCoh}_{gr,X}Y$ similarly.
\end{dfn}
\begin{rem}
Let $X=\operatorname{Spec}A$ be an affine scheme of finite type over $k$, and let $L,L' \in \operatorname{dgLie}_X^*$ and $L \to L'$ be a morphism. Consider the induced morphism $f:\operatorname{Spec}C_A(L) \to \operatorname{Spec}C_A(L')$. Proposition \ref{prop:lie} shows that the following functor
\begin{align*}
\D^+(U_A(L)) \to \operatorname{IndCoh}_X^+C_A(L) \to \operatorname{IndCoh}^+C_A(L)
\end{align*}
is an equivalence, and in particular $\operatorname{IndCoh}_X^+C_A(L) \cong \operatorname{IndCoh}^+C_A(L)$. Hence $f^!(A) \in \operatorname{IndCoh}^+C_A(L) \cong \operatorname{IndCoh}_X^+C_A(L)$, and thus $f^!$ preserves $\operatorname{IndCoh}_XC_A(L)$.
\end{rem}
Now, we have the following equivalences between dg-Lie algebras and corresponding derived formal stacks.
\begin{thm}\label{thm:afflie}
Let $X$ be an affine scheme of finite type over $k$, and let $\mathfrak{g}$ be a dg-Lie algebra over $X$. Let $Y=\mathcal{F}_X(\mathfrak{g})$ be a derived formal stack correspoinding to the dg-Lie algebra $\mathfrak{g}$. Then we have the following equivalences of categories:
\begin{align*}
\D(U_A(\mathfrak{g})) &\cong \operatorname{IndCoh}_X(Y), \\
\D^+(U_A(\mathfrak{g})) &\cong \D^+(Y).
\end{align*}
Moreover, if $\mathfrak{g}$ is a graded dg-Lie algebra, we also have
\begin{align*}
\D_{gr}(U_A(\mathfrak{g})) &\cong \operatorname{IndCoh}_{gr,X}(Y), \\
\D^+_{gr}(U_A(\mathfrak{g})) &\cong \D^+_{gr}(Y).
\end{align*}
\end{thm}
\begin{proof}
We only prove the ungraded case. First we write $\mathfrak{g}$ as the following colimit
\begin{align*}
\mathfrak{g} \cong \operatorname*{colim}_{L \in K} L,
\end{align*}
where $K$ is a sifted diagram in $\operatorname{dgLie}_A^*$. Then we have
\begin{align*}
\D(U_A(\mathfrak{g})) \cong \D(U_A(\operatorname*{colim}_{L\in K} L))\cong \lim_{L \in K} \D(U_A(L)).
\end{align*}
Here, for $L,L' \in K$ and $L \to L'$, the functor $\D(U_A(L)) \to \D(U_A(L'))$ is given by restriction of scalars. On the other hand, we have
\begin{align*}
\operatorname{IndCoh}_XY \cong \lim_{L \in K} \operatorname{IndCoh}_XC_A(L),
\end{align*}
where for $L,L' \in K$ and $L \to L'$, $\operatorname{IndCoh}_XC_A(L) \to \operatorname{IndCoh}_XC_A(L')$ is given by !-pullback functor.
We will construct a functor $\lim_{L \in K} D(U_A(L)) \to \lim_{L \in K} \operatorname{IndCoh}_XC_A(L)$ by gluing the equivalences constructed in the proof of Proposition \ref{prop:lie}. Consider the following adjunction
\begin{align*}
\theta : \operatorname{Alg}_k 
\mathrel{\substack{
  \longrightarrow\\[-0.7ex]
  \perp\\[-0.7ex]
  \longleftarrow
}}
(\operatorname{Pr}^L_k)_{\D(k)/} : E
\end{align*}
used in the proof of Theorem \ref{thm:loop}. Consider a diagram in $(\operatorname{Pr}^L_k)_{\D(k)/}$ given by
\begin{align*}
K \ni L \longmapsto (\operatorname{IndCoh}_XC_A(L),A) \in (\operatorname{Pr}^L_k)_{\D(k)/},
\end{align*}
with pushforward functors, the left adjoints of !-pullback functors.

Now, we have $E(\operatorname{IndCoh}_XC_A(L),A) \cong E(\operatorname{IndCoh}_XC_A(L),T) \cong \operatorname{REnd}_{C_A(K)}(T) \cong U_A(L)$ and hence we obtain a colimit-preserving functor $\D(U_A(L)) \to \operatorname{IndCoh}_XC_A(L)$ which sends $U_A(L)$ to $T$. This functor coincides with the one in Proposition \ref{prop:lie}, and thus it induces an equivalence $\D(U_A(L)) \cong \operatorname{IndCoh}_XC_A(L)$.

Now, the counit map induces a natural transformation between diagrams in $(\operatorname{Pr}^L_k)_{\D(k)/}$
\[
\begin{tikzcd}
K \arrow[r, bend left=50, "\D(U_A(L))"{name=F,above}] \arrow[r, bend right=50, "\operatorname{IndCoh}_XC_A(L)" {name=G,below}] \arrow[shorten <=10pt,shorten >=10pt,Rightarrow, from=F, to=G, ""] &\ \ \  (\operatorname{Pr}^L_k)_{\D(k)/}.
\end{tikzcd}
\]
Then, taking right adjoint of this, we obtain a functor
\begin{align*}
\D(U_A(\mathfrak{g})) \cong \lim_{L\in K} \D(U_A(L)) \longrightarrow \lim_{L\in K} \operatorname{IndCoh}_X(C_A(L)) \cong \operatorname{IndCoh}_XY,
\end{align*}
which is an equivalence. Restricting to the bounded-below subcategories, we also have
\begin{align*}
\D^+(U_A(\mathfrak{g})) \cong \D^+(Y).
\end{align*}
\end{proof}
Gluing these equivalences, we obtain the following result.
\begin{thm}\label{thm:main2}
Let $X$ be a separated scheme of finite type over $k$, and let $\mathfrak{g}$ be a dg-Lie algebra over $X$. Let $Y=\mathcal{F}_X(\mathfrak{g})$ be a derived formal stack correspoinding to the dg-Lie algebra $\mathfrak{g}$. Then we have the following equivalences of categories:
\begin{align*}
\D(U_X(\mathfrak{g})) &\cong \operatorname{IndCoh}_X(Y), \\
\D^+(U_X(\mathfrak{g})) &\cong \D^+(Y).
\end{align*}
Moreover, if $\mathfrak{g}$ is a graded dg-Lie algebra, we also have
\begin{align*}
\D_{gr}(U_X(\mathfrak{g})) &\cong \operatorname{IndCoh}_{gr,X}(Y),\\
\D^+_{gr}(U_X(\mathfrak{g})) &\cong \D^+_{gr}(Y).
\end{align*}
\end{thm}
\begin{proof}
We only prove the ungraded case. We will construct a functor
\begin{align*}
\D(U_X(\mathfrak{g}))\cong \lim_{U \in \operatorname{AffOp}(X)} \D(U_X(\mathfrak{g})|_U) \longrightarrow \lim_{U \in \operatorname{AffOp}(X)} \operatorname{IndCoh}_UY_U \cong \operatorname{IndCoh}_XY
\end{align*}
by gluing the equivalences constructed in the proof of Theorem \ref{thm:afflie}. Again, consider the following adjunction
\begin{align*}
\theta : \operatorname{Alg}_k 
\mathrel{\substack{
  \longrightarrow\\[-0.7ex]
  \perp\\[-0.7ex]
  \longleftarrow
}}
(\operatorname{Pr}^L_k)_{\D(k)/} : E
\end{align*}
used in the proof of Theorem \ref{thm:loop}. Let $F_U:\D(U_X(\mathfrak{g})|_U) \to \operatorname{IndCoh}_UY_U$ be the functor which is constructed in the proof of Theorem \ref{thm:afflie}, and we define $P_U := F_U(U_X(\mathfrak{g})|_U) \in \operatorname{IndCoh}_UY_U$. Then $\operatorname{End}(P_U) \cong U_X(\mathfrak{g})|_U$, and hence $E(\operatorname{IndCoh}_UY_U, P_U) \cong U_X(\mathfrak{g})|_U$. Moreover, 
\begin{align*}
\operatorname{AffOp}(X) \ni U \longmapsto (\operatorname{IndCoh}_UY_U , P_U) \in (\operatorname{Pr}^L_k)_{\D(k)/}
\end{align*}
gives a diagram in $(\operatorname{Pr}^L_k)_{\D(k)/}$. Indeed, for affine open sets $V \subset U \subset X$, $U \cong \operatorname{Spec} A$, we have
\begin{align*}
\operatorname{IndCoh}_U Y_U &\cong \lim_{L \in K} \operatorname{IndCoh}_U C_A(L),\\
\operatorname{IndCoh}_V Y_V &\cong \lim_{L \in K} \operatorname{IndCoh}_V C_A(L)|_V,
\end{align*}
where $K$ is a sifted diagram in $\operatorname{dgLie}_A^*$. The object $P_U$ is mapped to $T_L \otimes_{U_A(L)} U_X(\mathfrak{g})|_U$ by the functor $\operatorname{IndCoh} Y_U \to \operatorname{IndCoh} C_A(L)$ and
\begin{align*}
(T_L \otimes_{U_A(L)} U_X(\mathfrak{g})|_U)|_V \cong T_{L|_V} \otimes_{U_A(L|_V)} (U_X(\mathfrak{g})|_V),
\end{align*}
where $T_L$ denotes the bimodule defined in (4.2). Thus we have $P_U|_V \cong P_V$. Here, note that !-pullback and ordinary pullback along open immersion coincide through the functor $\operatorname{IndCoh}C_A(L) \to \D(C_A(L))$.

Hence, using the counit map, we can construct a funtor
\begin{align}
\lim_{U\in \operatorname{AffOp}(X)} \D(U_X(\mathfrak{g})|_U)) \cong \lim_{U\in \operatorname{AffOp}(X)} \D(\operatorname{End}(P_U)) \longrightarrow \lim_{U\in \operatorname{AffOp}(X)} \operatorname{IndCoh}_U Y_U.
\end{align}
Since the functor $\D(U_X(\mathfrak{g})|_U) \to \operatorname{IndCoh}_UY_U$ induced by the counit map is a colimit-preserving functor which sends $U_X(\mathfrak{g})|_U$ to $P_U$, this functor coincides with the functor constructed in the proof of Theorem \ref{thm:afflie},
and hence the functor (4.3) gives an equivalence. Thus we obtain an equivalence
\begin{align*}
\D(U_X(\mathfrak{g})) \overset{\sim}{\to} \operatorname{IndCoh}_UY.
\end{align*}
 Restricting this to bounded-below subcategories, we also have an equivalence
\begin{align*}
\D^+(U_X(\mathfrak{g})) \overset{\sim}\to \operatorname{IndCoh}_U^+ Y \cong \D^+(Y).
\end{align*}
\end{proof}
Combining Theorem \ref{thm:loop} and Theorem \ref{thm:main2}, we can relate derived loop spaces and derived formal stacks as follows.
\begin{cor}\label{main3}
Let $X$ be a separated scheme of finite type over $k$, and let $n \in \mathbb{Z}_{>0}$. Let $Y=\mathcal{F}_X(\mathbb{T}_{X/k}[-n-1])$ be a derived formal stack correspoinding to the commutative dg-Lie algebra $\mathbb{T}_{X/k}[-n-1]$ over $X$. Then we have the following equivalences of $\infty$-categories:
\begin{align*}
\operatorname{IndCoh}_X\mathcal{L}^{(n)}X &\cong \operatorname{IndCoh}_XY, \\
\operatorname{IndCoh}_{gr,X}\mathcal{L}^{(n)}X &\cong \operatorname{IndCoh}_{gr,X}Y,\\
\D^+(\mathcal{L}^{(n)}X) &\cong \D^+(Y),\\
\D^+_{gr}(\mathcal{L}^{(n)}X) &\cong \D^+_{gr}(Y).
\end{align*}
Moreover, if we define $Y'=\mathcal{F}_X(\mathbb{T}_{X/k})$ by regarding $\mathbb{T}_{X/k}$ as a commutative dg-Lie algebra over $X$, then we also have the following equivalence:
\begin{align*}
\operatorname{IndCoh}_{gr,X}\mathcal{L}X &\cong \operatorname{IndCoh}_{gr,X}Y'.
\end{align*}
\end{cor}
\begin{proof}
If we regard $\mathbb{T}_{X/k}[-n-1]$ as a commutative dg-Lie algebra, we have $U_X(\mathbb{T}_{X/k}[-n-1]) = \Sym(\mathbb{T}_{X/k}[-n-1])$. Then we obtain
\begin{align*}
\operatorname{IndCoh}_X\mathcal{L}^{(n)}X \cong \D(\Sym(\mathbb{T}_{X/k}[-n-1])) \cong \D(U_X(\mathbb{T}_{X/k}[-n-1])) \cong \operatorname{IndCoh}_XY,
\end{align*}
by Theorem \ref{thm:loop} and Theorem \ref{thm:main2}. Similarly, we also have
\begin{align*}
\operatorname{IndCoh}_{gr,X}\mathcal{L}X \cong \D_{gr}(\Sym (\mathbb{T}_{X/k}[-2])) \cong \D_{gr}(\Sym \mathbb{T}_{X/k}) \cong \D_{gr}(U_X(\mathbb{T}_{X/k})) \cong \operatorname{IndCoh}_{gr,X}Y'.
\end{align*}
\end{proof}

\end{document}